\documentclass[12pt]{amsart}

\usepackage{bookman} 
\usepackage{paralist, mathdots} 

\usepackage{amssymb} 
\usepackage{amsmath,tikz,wasysym,multirow,enumitem,lscape,multirow,subfigure} 
\usepackage{mwe}
\usepackage{hyperref} 
\usepackage{amsthm}
\usepackage{mathtools}
\usepackage{soul}
\usepackage{cite}

\tikzset{every node/.style={draw, circle, inner sep=2pt},
every picture/.append style={thick,scale=0.8},
every label/.style={draw=none, rectangle}}

\newtheorem{theorem}{Theorem}[section]

\newtheorem{corollary}[theorem]{Corollary}
\newtheorem{lemma}[theorem]{Lemma}

\theoremstyle{definition}
\newtheorem{example}[theorem]{Example}

\newtheorem{definition}[theorem]{Definition}

\newcommand{\npmatrix}[1]{\left( \begin{matrix} #1 \end{matrix} \right)}
\newcommand{\R}{\mathbb{R}}

\newcommand{\GG}{\mathcal{G}^{\rm SSP}}
\newcommand{\GGSAP}{\mathcal{G}^{\rm SAP}}
\newcommand{\GGSMP}{\mathcal{G}^{\rm SMP}}

\newcommand{\rk}{\operatorname{rk}}  
\newcommand{\tr}{\operatorname{tr}}

\newcommand{\oml}{{\bf m}}

\newcommand{\trans}{^\top}

\newcommand{\dunion}{\mathbin{\dot{\cup}}}

\newcommand{\bx}{{\bf x}}
\newcommand{\by}{{\bf y}}
\newcommand{\bz}{{\bf z}}

\newcommand{\bzero}{{\bf 0}}
\renewcommand{\S}{\mathcal{S}}
\newcommand{\Szb}{\overline{\mathcal{S}_0}}

\newcommand{\dist}{\operatorname{dist}}
\newcommand{\forces}[2]{\xrightarrow[#1]{#2}}

\makeatletter
\newcommand{\manuallabel}[2]{\def\@currentlabel{#2}\label{#1}}
\makeatother

    \definecolor{helena}{rgb}{.2,.8,.4}
    \definecolor{polona}{rgb}{.8,.2,.2}
    \definecolor{jephian}{rgb}{.3,.4,.1}
   \definecolor{todo}{rgb}{.2,.2,.8}

\begin{document}
\title{The strong spectral property for graphs}
\author{Jephian C.-H.~Lin\and Polona Oblak\and Helena \v{S}migoc}
\address[J.~C.-H.~Lin]{Department of Applied Mathematics, National Sun Yat-sen University, Kaohsiung 80424, Taiwan}
\email{jephianlin@gmail.com}
\address[P.~Oblak]{Faculty of Computer and Information Science, University of Ljubljana, Ve\v cna pot 113, SI-1000 Ljubljana, Slovenia}
\email{polona.oblak@fri.uni-lj.si}
\address[H.~\v{S}migoc]{School of Mathematics and Statistics, University College Dublin, Belfield, Dublin 4, Ireland}
\email{helena.smigoc@ucd.ie}

\subjclass[2010]{05C50, 15A18, 15B57}
\keywords{Symmetric matrix; Inverse Eigenvalue Problem; Strong spectral property; Graph}
\maketitle

\begin{abstract}
 We introduce the set $\GG$ of all simple graphs $G$ with the property that each symmetric matrix corresponding to a graph $G \in \GG$ has the strong spectral property. We find several families of graphs in $\GG$ and, in particular,  characterise the trees in $\GG$.
\end{abstract}

\section{Introduction}

Let $G=(V(G),E(G))$ be an undirected simple graph.  Define $\S(G)$ to be the set of all real symmetric $n \times n$ matrices $A=\begin{bmatrix}a_{ij}\end{bmatrix}$ such that $a_{ij}\neq 0$ if and only if $\{i,j\} \in E(G)$, with no restriction on the diagonal entries of $A$.  The \emph{inverse eigenvalue problem of a graph} (IEPG) aims to find all  possible spectra of matrices in $\S(G)$. The problem exposes the relations between graphs and matrices, and is motivated by theory of vibrations; see \cite{MR2102477} and the references therein.  The  study of the IEPG has been an active area of research \cite{MR3665573,MR2350678,MR3013937,MR0244285,MR3118942,MR3291662} that motivated several related questions on the spectral properties of matrices with a given pattern, such as the minimum rank problem \cite{MR2350678,MR2388646}, and the problem on the minimum number of distinct eigenvalues \cite{MR3118943,2017arXiv170800064B, MR3891770}.

Strong properties are one of the most powerful tools in the study of the IEPG.  The strong Arnold property  is the first strong property that was introduced \cite{MR1224700,MR1070462}, and it appeared in the definition of the Colin de Verdi\`ere parameter $\mu(G)$ as a non-degenerate condition that allows a matrix to be perturbed slightly without changing its rank.
In \cite{MR3665573,2017arXiv170800064B} the strong multiplicity property (the SMP) and the strong spectral property (the SSP) were defined. Those properties allow a matrix to be perturbed slightly without changing its ordered multiplicity list and its spectrum. 

Whether a matrix has the strong properties depends on the entries of the matrix. However, in some special cases the pattern conditions coming from the graph are enough to guarantee a strong property of a matrix. This paper aims to identify those special cases, i.e.\ those graphs for which all matrices in $\S(G)$ have the strong spectral property. 
We will give a more concise formulation of this question, as well as some further background to the problem, after we establish some notation.

%
%

\subsection{Strong Properties}

The family of all real symmetric matrices of order $n \times n$ will be denoted by $\S_n(\mathbb{R})$. 
  Suppose that a matrix $A\in \S_n(\R)$ has eigenvalues 
\[\lambda_1 < \lambda_2<\cdots<\lambda_{q}\] 
with multiplicities
 $m_1,m_2,\ldots, m_q$.  The \emph{spectrum} of $A$ is the set of all eigenvalues (counting multiplicities) of $A$ and the  \emph{ordered multiplicity list} of $A$ is defined to 
 be the sequence $\oml(A) = (m_1,m_2,\ldots,m_q)$.
 
 By $I_n$ and $0_n$ we  denote the identity and the zero 
matrix in $\S_n(\R)$, respectively. 
 For matrices $A$ and $B$ of the same order, we  denote their entrywise product by $A\circ B$, and their commutator by $[A,B]=AB-BA$.

 The following notions for a symmetric matrix $A$ were defined in  \cite{MR3665573,MR1224700}:
\begin{itemize}
\item $A$ has the \emph{strong spectral property} (the \emph{SSP}) if 
 the only symmetric matrix $X$ satisfying
 $$A\circ X=I \circ X=0 \, \text{ and }\, [A,X]=0$$
 is the zero matrix.
\item $A$ has the \emph{strong multiplicity property} (the \emph{SMP}) if 
the only symmetric matrix $X$ satisfying
 $$A\circ X=I \circ X=0, [A,X]=0 \, \text{ and }\, \tr(A^iX)=0$$
 for $i=2,3,\ldots,q(A)$,  is the zero matrix.
\item  $A$ has the \emph{strong Arnold property} (the \emph{SAP}) if 
the only symmetric matrix $X$ satisfying
 $$A\circ X=I \circ X=0 \, \text{ and }\, AX=0$$
  is the zero matrix.
\end{itemize}
By definition, the SSP implies the SMP, and the SMP implies the SAP.

The SSP and the SMP have already been exploited in several papers  \cite{2017arxiv170802438, MR3067819,MR2181887, 2017arXiv170800064B, MR3665573, 2017arXiv170801821B,MR3891770, MR3536955,MR1703435,MR3621182,MR2009404}.
To illustrate the value of those properties, we highlight the following results from \cite{MR3665573}:
\begin{itemize}
\item If $A\in \S(G)$ has the SSP and $G$ is a spanning subgraph of $H$, then there 
exists $B\in \S(H)$ with the SSP and  $\sigma(A)=\sigma(B)$.
\item If $A\in \S(G)$ has the SMP and $G$ is a spanning subgraph of $H$, then there 
exists $B\in \S(H)$ with the SMP and  $\oml(A)=\oml(B)$.
\end{itemize}

A natural question, motivated by the results above, is to find ways to produce matrices with the strong properties. Clearly, whether a matrix has any of the strong properties depends on the entries of the matrix, not only on its pattern. However, in some very special cases, a strong property may be apparent from the pattern of the matrix. (As a trivial example, any matrix $A \in \S(K_n)$ has the SSP.) 

Some work in this direction has been done on the SAP.  Schrijver and Sevenster \cite{MR3621182} showed that if $G$ is a $4$-connected flat graph and $A$ is a nonpositive matrix in $\S(G)$ with exactly one negative eigenvalue, then $A$ has the SAP.  Lin \cite{MR3536955} showed that if a graph $G$ has the combinatorial property $Z_{\rm SAP}(G)=0$, then every matrix in $\S(G)$ has the SAP. 
In this manuscript we aim to extend this investigation to the SMP and most notably to the SSP.

With this question in mind, we define $\GG$ to be the set of all graphs $G$  such that all matrices $A\in \S(G)$ have the SSP. The sets $\GGSMP$ and $\GGSAP$ are defined in a similar way. Note that $\GG \subseteq \GGSMP \subseteq \GGSAP$.



Our paper offers several methods to determine the strong spectral property from the structure of a graph, or equivalently, to find graphs in $\GG$.  Section~\ref{sec:nssp} lists some results that allow us to identify graphs that do not have the strong spectral property.  In Section~\ref{Sec3} we develop several methods that give sufficient conditions for a  graph to be in $\GG$.  In Section~\ref{sec:trees}, the characterization of trees in $\GG$ is provided.  Finally, a selection of further examples of graphs in $\GG$ (or not in $\GG$) is given in Section~\ref{SSPgraphs}. 

 \subsection{Notation}\label{notation}
 
 
 At this point we will establish some basic notation concerning graphs. For a graph $G=(V(G),E(G))$ we denote its order by $|G|=|V(G)|$.  
We say that two distinct vertices $x,y\in V(G)$ are \emph{adjacent} in~$G$ if $\{x,y\}\in E(G)$.  A sequence of~$k$ edges $\{x_0,x_1\}, \{x_1,x_2\},\ldots,\{x_{k-1},x_{k}\}$ from $E(G)$ with $x_1,\ldots,x_k$ all distinct is called 
\emph{a path of length $k$} in $G$ between the vertices~$x_0$ and $x_k$. The \emph{distance} between two vertices in a graph $G$ is the length (the number of edges) of the shortest path in $G$ between them.
By $N_{G}[v]$ we  denote the \emph{closed neighbourhood} of vertex $v \in G$, i.e. the set of all vertices in $G$ at distance at most 1 from $v$. For $W \subseteq V(G)$, let $G[W]$ denote the induced subgraph of $G$ on vertices $W$.


We will use standard notation for the most common families of graphs. Therefore, $K_n$ will denote the complete graph on $n$ vertices, and $K_{m,n}$ is the complete bipartite graph whose two parts have $m$ and $n$ vertices, respectively. Furthermore, $P_n$ will denote the path on $n$ vertices, and $C_n$ the cycle on $n$ vertices.


There are several established operations on graphs. Here we list the ones that we will meet in this paper.  
\begin{itemize}
\item The \emph{complement}  $G^c$ of a graph $G$ is the graph on vertices $V(G)$ such that two distinct vertices are adjacent in $G^c$ if and only if they are not adjacent in $G$.  
\item  The \emph{$r$-th power} of $G$, denoted as $G^r$, is the graph obtained from $G$ by adding every pair $\{i,j\}$ at a distance at most $r$ in $G$ as an edge. 
\item  The \emph{$r$-th strong power} of $G$, denoted as $G^{(r)}$, is the graph obtained from $G$ by adding every pair 
$\{i,j\}$ at a distance $r$ in $G$ as an edge.
\item The \emph{join} $G \vee H$ of $G$ and $H$ is the graph union $G \cup H$ together with all the possible edges joining the vertices in $G$ to the vertices in $H$.
\item   The \emph{tensor product} of two graphs is defined in the following way:
$G \times H=(V(G \times H), E(G \times H))$, where $V(G \times H)=V(G) \times V(H)$
and $((u,u'),(v,v')) \in E(G \times H)$ if and only if $ (u,v) \in E(G)$ and $(u',v') \in E(H)$.
\item If $V(G)=\{v_1,\ldots,v_n\}$, then the \emph{corona graph} $G\odot H$
of graphs $G$ and $H$ is defined as the graph obtained by taking $n$ copies of $H$ and for each $i$ adding edges between the $i$-th vertex $v_i$ of 
$G$ to each vertex of the $i$-th copy of $H$.
\item If $G$ and $H$ are graphs with $V(H) = V(G)$ and $E(G)= E(H) \cup \{e_1,\ldots, e_k\}$, we write $G=H+\{e_1,\ldots, e_k\}$. If $k=1$, we abbreviate this notation to  $G=H+e_1$.
\end{itemize}
  
 
By $E_{ij}$ we will denote the $0$-$1$ matrix with 
the only nonzero element in the $(i,j)$-th position, where the size of the matrix will be clear from the context.

If $A=(a_{ij})\in \R^{m \times n}$ and $B\in \R^{p \times q}$, then the \emph{tensor} or the \emph{Kronecker product} $A\otimes B$
of matrices $A$ and $B$ is the $mp \times nq$ block matrix, with the $(i,j)$-th block equal to $ a_{ij} B$.

 
 The adjacency matrix of a graph $G$ is the matrix $A\in \S(G)$ defined by
 \[A_{ij}=\begin{cases}
  1, & \{i,j\} \in E(G),\\
  0, & \{i,j\} \notin E(G) \text{ or } i=j.\\
 \end{cases}\]
 Note that the adjacency matrix of a tensor product of graphs is the tensor product of their adjacency matrices.

 Let $\Szb(G)$ be the set of matrices with zero diagonal entries in the closure of $\S(G)$; that is, $\Szb(G)$ are those real symmetric $n\times n$ matrices whose $(i,j)$-entry is zero whenever $i=j$ or $\{i,j\}\notin E(G)$.
 
   For a matrix $A\in\S_n(\mathbb{R})$, let $q(A)$ be the number of distinct eigenvalues of $A$. For a graph $G$, the \emph{minimum number of distinct eigenvalues of} $G$ is defined as 
\[q(G) = \min\{q(A): A \in \S(G)\}.\]

\section{Conditions for graphs not to be in $\GG$}
\label{sec:nssp}

The condition $G \in \GG$ is very restrictive, and as one would expect, most graphs are not contained in $\GG$. In this section we give some general results that imply $G \notin \GG$.  
Our first example shows that any regular graph $G$, not equal to $K_n$, is not contained in $\GG$.

\begin{example}\label{cycle}
Let $G \neq K_n$ be a regular graph, $A\in \S(G)$ be the adjacency matrix of $G$, 
$J$ be the $n\times n$ all-ones matrix, and $X=J-A-I_n$. Since $G$ is a regular graph, we have $AJ=JA$. Now it is straightforward to check that $X$ is a nonzero matrix satisfying: 
$A \circ X=0$, $I_n \circ X=0$,  and $[A,X]=0$, hence proving that $G \notin \GG$.  
\end{example}


%


Once we have identified some graphs not in $\GG$, one can construct other graphs with the same property. We give some examples below.

\begin{lemma} 
Let  $A, X \in \S_m(\R)$ and  $S_i,T \in \S_n(\R)$ for $i=0,1,\ldots,n-1$ be matrices such that:
\begin{itemize} 
\item $A \circ X=I_m \circ X=0$ and $[A,X]=0$, 
\item  $[S_i,T]=0$ for $i=0,\ldots,n-1$, and $S_i \circ T=0$ for $i=2, \ldots,n-1$.
\end{itemize}
Then the matrices $\hat A=\sum_{j=0}^{n-1}(S_j \otimes A^j)$ and $\hat X=T \otimes X$ satisfy 
\[\hat{A}\circ \hat{X}= I_{mn} \circ \hat{X}=0\text{ and }[\hat A,\hat X]=0\]
for any nonnegative integer $k$.
\end{lemma}

\begin{proof}
The proof is a straightforward application of properties of operations $\otimes$ 
and $\circ$. Indeed, we have 
\begin{align*}
\hat{A}\circ \hat{X}&=\left(\sum_{j=0}^{n-1}(S_j \otimes A^j)\right) \circ (T \otimes X)\\
&=\sum_{j=0}^{n-1}(S_j \circ T) \otimes (A^j \circ X) =0,
\end{align*}
since $S_j\circ T=0$ for $j=2,\ldots, n-1$ and $A^j\circ X=0$ for $j=0,1$.  Also,
\begin{align*}
I_{mn}\circ \hat{X}=(I_n\circ T)\otimes (I_m\circ X)=0,
\end{align*}
and
\begin{align*}
[\hat A, \hat X]&=\left(\sum_{j=0}^{n-1}(S_j \otimes A^j)\right)(T \otimes X)-(T \otimes X)\left(\sum_{j=0}^{n-1}(S_j \otimes A^j)\right)\\
&=\sum_{j=0}^{n-1}\left( (S_j T \otimes A^jX)-(TS_j \otimes XA^j) \right) \\
&=\sum_{j=0}^{n-1}\left( (S_j T \otimes A^jX)-(S_jT \otimes XA^j) \right) \\
&=\sum_{j=0}^{n-1}\left( (S_j T \otimes (A^jX-XA^j) \right)=0.
\end{align*}
Therefore, $\hat{A}$ and $\hat{X}$ have the desired properties.
\end{proof}

Note that for $T=I_n$, the condition  $[S_i,T]=0$  is automatically satisfied. In particular, this choice gives us the following corollary.  

\begin{corollary}
Let $G \notin \GG$. Then: 
\begin{enumerate}
\item  $G \otimes H \notin \GG$ for any graph $H$. 
\item The corona graph $G \odot K_{m-1}^c \notin \GG$. 
\end{enumerate}
\end{corollary}

\begin{proof}
Suppose $|G|=n$ and $|H|=m$.  Let $A\in\S(G)$ and $X\in\S_n(\R)$ be matrices satisfying $A\circ X=I_n\circ X=0$ and $[A,X]=0$.  To prove item 1), take $T=I_m$, $S_1$ to be the adjacency matrix of $H$, and $S_0=S_2= \cdots =S_{n-1}=0_m$. To prove item 2), take $T=I_m$, $S_0=\sum_{i=2}^n (E_{1j}+E_{j1})$, $S_1=E_{11}$, and $S_2= \cdots =S_{n-1}=0_m$.
\end{proof}

\begin{definition}
A \emph{barbell partition} of a graph $G$ is a partition of $V(G)$ into three disjoint parts $\{R,W_1,W_2\}$ such that
\begin{itemize}
\item $R$ is allowed to be an empty set, but $W_i\neq\emptyset$ for $i=1,2$;
\item there are no edges between vertices in $W_1$ and vertices in $W_2$;
\item for each $v\in R$, $|N_G(v)\cap W_i|\neq 1$ for $i=1,2$.
\end{itemize}
\end{definition}

As a trivial example, $2K_1$ has a barbell partition where $R=\emptyset$ and both $W_1,W_2$ consist of one element. On the other hand, any barbell partition for $K_{1,3}$ would have to have the center vertex included in $R$, to assure that there are no edges between $W_1$ and $W_2$. This would imply $|W_1|=1$ or $|W_2|=1$, contradicting the condition that the center vertex cannot have exactly one neighbour in $W_1$ or $W_2$. 

\begin{lemma}\label{lem:par}
  Let $G$ be a graph with a barbell partition.  Then there is a matrix $M\in \S(G)$ such that $M$ does not have the SAP (and the SSP).
\end{lemma}
\begin{proof}
Let $\{R,W_1,W_2\}$ be a barbell partition of $G$.  Let us define a matrix $M\in \S(G)$ as 
\[M=\begin{bmatrix} 
A & B_1\trans  & B_2\trans  \\
B_1 & L_1 & 0 \\
B_2 & 0 & L_2
\end{bmatrix},\]
where the block partition of $M$ is consistent with the partition $\{R,W_1,W_2\}$, $A$ is the adjacency matrix of $G[R]$, $L_i$ are the Laplacian matrices of $G[W_i]$ for $i=1,2$, and $B_1$ and $B_2$ are matrices having the required pattern and zero column sums. Note that the condition $|N_G(v)\cap W_i|\neq 1$ for each $v \in R$, guarantees that such matrices $B_i$ exist. 

Let 
\[X=\begin{bmatrix}
0 & 0 & 0 \\
0 & 0 & J \\
0 & J^{\trans} & 0
\end{bmatrix},\]
where $J$ denotes a matrix of appropriate size with all its elements equal to $1$. 
Clearly, we have $M\circ X=I\circ X=0$. Moreover, $MX=0$, since all the matrices $B_1$, $B_2$, $L_1$, and $L_2$ have zero column sums.  This proves that the matrix $M$ does not have the SAP, so it does not have the SSP.  
\end{proof}

Let us look at a special case of this Lemma, when $R$ is a single vertex.

\begin{corollary}\label{cor:deg4}
Let $G$ be a graph with a cut-vertex $v$ and $G-\{v\}=G_1\cup G_2$ such that $v$ has at least 2 neighbours in each $G_i$. Then $G \notin \GG$. (Note that $G_i$ are not required to be connected.)
\end{corollary}

\begin{proof}
Let $R=\{v\},$ $W_1=V(G_1)$ and $W_2=V(G_2)$. Then $\{R, W_1, W_2\}$ is a barbell partition of $G$. By 
Lemma \ref{lem:par} it follows that $G \notin \GG$. 
\end{proof}

\begin{corollary}\label{cor:0-2trees}
If a graph $G$ contains a vertex of degree at least four contained in at most one cycle, then $G \notin \GG$.
\end{corollary}

For trees and uni-cyclic graphs, we can make this observation even more specific. 

\begin{corollary}\label{rem:tnotinGG}
A tree $T$ with a vertex $v$ satisfying $\deg(v)\geq 4$ or two  vertices $u,v$ with $\deg(u),\deg(v)\geq 3$ is not in $\GG$.   
\end{corollary}

\begin{proof}
It is easy to see that a tree $T$ satisfying conditions above has a barbell partition so $T\notin\GG$ by Lemma~\ref{lem:par}.
\end{proof}

\begin{corollary}\label{rem:uninotinGG}
A unicyclic graph with a  vertex  of degree at least four or  with a degree three vertex not contained in the cycle is not in $\GG$.  
\end{corollary}

\begin{proof}
If a unicyclic graph $G$ has a vertex of degree at least four, then by Corollary \ref{cor:0-2trees} we have $G \notin \GG$. If $G$ 
has  a degree three vertex not contained in the cycle, then it has a degree three vertex on the cycle as well.  Thus it has a  barbell partition, so $T\notin\GG$ by Lemma~\ref{lem:par}.
\end{proof}

\section{Local properties of graphs related to the SSP}
\label{Sec3}

Let $G$ be a graph on $n$ vertices, $A=\begin{bmatrix}a_{ij}\end{bmatrix}\in\S(G)$, and $X=\begin{bmatrix}x_{ij}\end{bmatrix}$ a symmetric matrix satisfying  
 $$A\circ X = I\circ X=0 \text{ and }[A,X] = 0.$$
The first two equalities guarantee that $x_{ij}=0$ for $i=j$ and for $\{i,j\} \in E(G)$. 
The condition $[A,X]=0$ gives several of equalities of the form: 
\begin{align}
   [A,X]_{i,j}&=\sum\limits_{k=1}^n (a_{ik}x_{kj}- a_{jk}x_{ik}) \nonumber \\
    &=\sum\limits_{k\in N_G[i]\cap N_{G}[j]^c} a_{ik}x_{kj}-
        \sum\limits_{k\in N_G[j] \cap N_{G}[i]^c} a_{jk}x_{ik}=0, \label{initcen}
\end{align}
that we will view as a system of linear equations in variables $x_{ij}$ for $\{i,j\}\in E(G^c)$. Or aim is to identify families of graphs $G$ for which this system has only a trivial solution for all choices of $A \in S(G).$ 

For example, it may happen that pattern constrains on $A$ and $X$ coming from $G$ cause an equation of the form \eqref{initcen} to consist of only one term for some specific $i$ and $j$: 
\[[A,X]_{i,j} = a_{ik}x_{kj} = 0.\]
If $\{i,k\}\in E(G)$, then $a_{ik}\neq 0$ and consequently $x_{kj}=0$. To keep an account of entries of $X$ that we  know are zero at each specific step, we define a new graph: $G_1 = G + \{k,j\}$, and note that  
$X\in\Szb(G_1^c)$. 
 Suppose a similar situation occurs consecutively with 
 \[G = G_0 \subset G_1 \subset \cdots \subset G_{r},\] 
and we can gradually deduce more entries in $X$ to be zero: 
\[X\in\Szb(G_0^c)\implies \cdots\implies X\in\Szb(G_{r-1}^c)\implies X\in\Szb(G_{r}^c),\]
then (\refeq{initcen}) is simplified to:
\begin{align}
   [A,X]_{i,j}&= \sum_{k\in N_G[i]\cap N_{G_r}[j]^c} a_{ik}x_{kj}-
        \sum_{k\in N_G[j] \cap N_{G_r}[i]^c} a_{jk}x_{ik}=0. \label{ceneqn}
\end{align}
If this process completes with $G_r=K_n$, we can conclude that $X\in\Szb(K_n^c)$ has to be the zero matrix and therefore $A$ has the SSP.  Notice that this argument is independent of the choice of $A\in\S(G)$, so this process, if it happens, guarantees $G\in\GG$.

To be able to identify equations that can help in this process we need to know how many terms or which terms occur in \eqref{ceneqn}.

\begin{definition}
Let $G$ be a graph and $G_l$ a supergraph of $G$ of the same order.  Let $i,j\in V(G)$ and $U\subseteq V(G)$ a nonempty set.  The pair $\{i,j\}$ is said to be \emph{focused} on $U$ 
with respect to $G$ and $G_l$ if 
\[N_G[i]\cap N_{G_l}[j]^c \subseteq U \text{ and } N_G[j] \cap N_{G_l}[i]^c\subseteq U.\]
\end{definition}

\begin{example}\label{ex:focused}
Suppose $G$ is the lollipop graph $L_{3,2}$ with 5 vertices as shown on Figure \ref{fig:L32}.
\begin{center}
 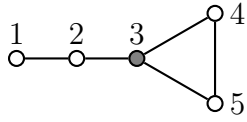
\begin{figure}[htb!]
 \centering
   \begin{tikzpicture}
	\node [label={above:$1$}] (1) at (-2,0) {};
	\node [label={above:$2$}] (2) at (-1,0) {};
	\node [fill=gray,label={above:$3$}] (3) at (0:0) {};
	\node [label={right:$4$}] (4) at (30:1.5) {};
	\node [label={right:$5$}] (5) at (330:1.5) {};
	\draw (3) -- (4) -- (5) -- (3)--(2) -- (1);
    \end{tikzpicture} 
    \caption{Lollipop graph $L_{3,2}$}\label{fig:L32}
    \end{figure}
    \end{center}
Note that 
$$N_G[1]\cap N_{G}[2]^c =\emptyset \text{ and } N_G[2] \cap N_{G}[1]^c=\{3\}$$
and thus $\{1,2\}$ is focused on $\{3\}$ with respect to $G$.
\end{example}

%

If a pair $\{i,j\}$ is focused on $U$, then the equation \eqref{initcen} simplifies to:
$$\sum\limits_{k\in U} (a_{ik}x_{kj}-a_{jk}x_{ik})=0.$$
In particular, if a pair of vertices is focused on a single vertex, like in Example \ref{ex:focused}, we obtain the following result.

 \begin{lemma}\label{thm:rule1}
Let $G$ be a graph and $G_l$ a supergraph of $G$ of the same order.  Suppose $A\in\S(G)$, 
$X\in\Szb(G_l^c)$, and $[A,X]=0$.  If for some distinct vertices $i$, $j$, and $k \in N_G[i]$ the pair 
$\{i,j\}$ is focused on $\{k\}$ with respect to $G$ and $G_l$, 
then $X\in\Szb(G_{l+1}^c)$ with $G_{l+1} = G_l + \{j,k\}$.
\end{lemma}
\begin{proof}
Since $\{i,j\}$ is focused on $\{k\}$ and $E(G) \subseteq E(G_l)$, we can assume without loss of  generality  that 
$$ N_{G}[i]\cap N_{G_l}[j]^c=\{k\} \text{ and } N_G[j] \cap N_{G_l}[i]^c=\emptyset.$$ By \eqref{ceneqn} we have
\[ [A,X]_{i,j} =a_{ik}x_{kj}=0. \]
Since $k\in N_G[i]$ and $a_{ik}\neq 0$, it follows that  $x_{kj}=0$.
\end{proof}
 
\begin{definition}
When the conditions of Lemma  \ref{thm:rule1} are satisfied and we update $G_l$ with $G_{l+1}$,  we say that the edge $\{i,j\}$ in $G$ forces the edge $ \{j,k\}$ to be added to $G_l$, and we denote this by 
$$\{i,j\}\forces{G}{G_l} \{j,k\}.$$  Notice that this notion is with respect to $G$ and $G_l$.  When the context is clear, we also write $\{i,j\}\rightarrow \{j,k\}$.  Also note that both $\{i,j\}$ and $\{j,k\}$ are unordered pairs, and $j$ is the element that occurs in both sets.
\end{definition}

We will give several applications of Lemma~\ref{thm:rule1} in Sections  \ref{sec:trees} and \ref{SSPgraphs}, at this point we offer a single example for illustration. 

\begin{example}
 Let us revisit the graph $G=G_0=L_{3,2}$ shown in Figure \ref{fig:L32}. 
 
 Since $\{1,2\}$ is focused on $\{3\}$ with respect to $G$ (see Example \ref{ex:focused}), we have $\{1,2\}\forces{G}{G} \{1,3\}$, and we define $G_1=G+\{1,3\}$. (See Figure \ref{fig:G0}.)
 
 \begin{center}
 \begin{figure}[htb!]
 \centering
 \subfigure	[$G=G_0$]{
      \begin{tikzpicture}
	\node [label={above:$1$}] (1) at (-2,0) {};
	\node [label={above:$2$}] (2) at (-1,0) {};
	\node [label={above:$3$}] (3) at (0:0) {};
	\node [label={right:$4$}] (4) at (30:1.5) {};
	\node [label={right:$5$}] (5) at (330:1.5) {};
	\draw (3) -- (4) -- (5) -- (3)--(2) -- (1);
    \end{tikzpicture} \label{fig:G0}}
\subfigure	[$G_1$]{
      \begin{tikzpicture}
	\node [label={above:$1$}] (1) at (-2,0) {};
	\node [label={above:$2$}] (2) at (-1,0) {};
	\node [label={above:$3$}] (3) at (0:0) {};
	\node [label={right:$4$}] (4) at (30:1.5) {};
	\node [label={right:$5$}] (5) at (330:1.5) {};
	\draw (3) -- (4) -- (5) -- (3)--(2) -- (1);
	\draw[color=green] (3) to[bend left] (1);
    \end{tikzpicture} \label{fig:G1}}
\subfigure	[$G_2$]{
      \begin{tikzpicture}
	\node [label={above:$1$}] (1) at (-2,0) {};
	\node [label={above:$2$}] (2) at (-1,0) {};
	\node [label={above:$3$}] (3) at (0:0) {};
	\node [label={right:$4$}] (4) at (30:1.5) {};
	\node [label={right:$5$}] (5) at (330:1.5) {};
	\draw (3) -- (4) -- (5) -- (3)--(2) -- (1);
	\draw[color=gray] (3) to[bend left] (1);
	\draw[color=green] (4) to[bend right] (2);
    \end{tikzpicture} \label{fig:G2}}
\subfigure	[$G_3$]{
      \begin{tikzpicture}
	\node [label={above:$1$}] (1) at (-2,0) {};
	\node [label={above:$2$}] (2) at (-1,0) {};
	\node [label={above:$3$}] (3) at (0:0) {};
	\node [label={right:$4$}] (4) at (30:1.5) {};
	\node [label={right:$5$}] (5) at (330:1.5) {};
	\draw (3) -- (4) -- (5) -- (3)--(2) -- (1);
	\draw[color=gray] (3) to[bend left] (1);
	\draw[color=gray] (4) to[bend right] (2);
	\draw[color=green] (5) to[bend left] (2);
    \end{tikzpicture} \label{fig:G3}}
    \subfigure	[$G_4$]{
      \begin{tikzpicture}
	\node [label={above:$1$}] (1) at (-2,0) {};
	\node [label={above:$2$}] (2) at (-1,0) {};
	\node [label={above:$3$}] (3) at (0:0) {};
	\node [label={right:$4$}] (4) at (30:1.5) {};
	\node [label={right:$5$}] (5) at (330:1.5) {};
	\draw (3) -- (4) -- (5) -- (3)--(2) -- (1);
	\draw[color=gray] (3) to[bend left] (1);
	\draw[color=gray] (4) to[bend right] (2);
	\draw[color=gray] (5) to[bend left] (2);
	\draw[color=green] (4) to[bend right] (1);
    \end{tikzpicture} \label{fig:G4}}
    \subfigure	[$G_5$]{
      \begin{tikzpicture}
	\node [label={above:$1$}] (1) at (-2,0) {};
	\node [label={above:$2$}] (2) at (-1,0) {};
	\node [label={above:$3$}] (3) at (0:0) {};
	\node [label={right:$4$}] (4) at (30:1.5) {};
	\node [label={right:$5$}] (5) at (330:1.5) {};
	\draw (3) -- (4) -- (5) -- (3)--(2) -- (1);
	\draw[color=gray] (3) to[bend left] (1);
	\draw[color=gray] (4) to[bend right] (2);
	\draw[color=gray] (5) to[bend left] (2);
	\draw[color=gray] (4) to[bend right] (1);
	\draw[color=green] (5) to[bend left] (1);
    \end{tikzpicture} \label{fig:G5}}
\caption{Sequence of graphs $G=G_0 \subseteq G_1 \subseteq \ldots \subseteq G_5=K_5$, defined in Lemma \ref{thm:rule1}. The new edges in $G_{k+1}$ not in $G_{k}$ are colored green.}
    \end{figure}
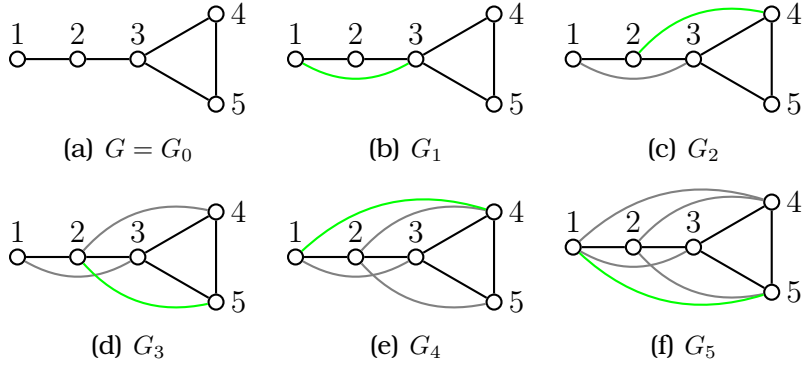
    \end{center}
    
    Now, $\{3,4\}$ is focused on $\{2\}$  with respect to $G$ and $G_1$ since
     $$N_G[3]\cap N_{G_1}[4]^c =\{2\} \text{ and } N_G[4]\cap N_{G_1}[3]^c =\emptyset.$$
     By Lemma \ref{thm:rule1} we have  $\{3,4\}\forces{G}{G_1} \{2,4\}$ and $G_2=G_1+\{2,4\}$. Similarly $\{3,5\}$ is focused on $\{2\}$  with respect to $G$ and $G_1$ and
     $\{3,5\}\forces{G}{G_1} \{2,5\}$ and so we may define $G_3=G_2+\{2,5\}$. (See Figures \ref{fig:G2} and \ref{fig:G3}.)
     
     Observe that $\{2,4\}$ is focused on $\{1\}$  with respect to $G$ and $G_3$ since
     $$N_G[2]\cap N_{G_3}[4]^c =\{1\} \text{ and } N_G[4]\cap N_{G_3}[2]^c =\emptyset.$$
     Therefore  $\{2,4\}\forces{G}{G_3} \{1,4\}$ and $G_4=G_3+\{1,4\}$. Similarly $\{2,5\}$ is focused on $\{1\}$  with respect to $G$ and $G_3$ and
     $\{2,5\}\forces{G}{G_3} \{1,5\}$ and so we define $G_5=G_4+\{1,5\}=K_5$. (See Figures \ref{fig:G4} and \ref{fig:G5}.)

    Using Lemma \ref{thm:rule1}  we have proved that $X\in\Szb(G_{5}^c)=\Szb(5 K_1)$, hence $X=0$. This implies that every matrix $A\in \S(L_{3,2})$ has the SSP, hence $L_{3,2}\in \GG$.
\end{example}

Lemma \ref{thm:rule1} identifies a very simple condition. As illustrated in the example above, its complexity comes from carrying it over several steps. One of the most straightforward examples, where the result can be useful are induced paths in graphs.  

\begin{corollary}\label{cor:rule1cor}
Let $G$ be a graph and $G_l$ a supergraph of $G$ of the same order.  Suppose a path $P$ is an induced subgraph of $G$ with vertices labeled as $p_1,\ldots,p_n$ in the path order such that $E(P^d)\subseteq E(G_l)$ for some $d$. If $\{p_{i},p_{i+d}\}$ is focused on $V(P)$
for $i=1,\ldots,n-d-1$ with respect to $G$ and $G_l$, then 
\[\{p_i,p_{i+d}\}\rightarrow \{p_i,p_{i+d+1}\}\]
may occur sequentially and the resulting graph $G_{l_1}$ is 
$$G_{l_1}=G_l+ \{\{p_i,p_{i+d+1}\}; 1 \leq i \leq n-d-1\}.$$
\end{corollary}

\begin{proof}
Suppose some of $\{p_i,p_{i+d}\}\rightarrow\{p_i,p_{i+d+1}\}$ has been applied so that the current $G_{l'}$ contains the edges $\{p_i,p_{i+d+1}\}$ with $i=1,\ldots, s-1$ for some $s$.  If $\{p_s,p_{s+d+1}\}\notin G_{l'}$, we show that 
\[\{p_s,p_{s+d}\}\forces{G}{G_{l'}} \{p_s,p_{s+d+1}\}.\]
 
Since $\{p_s,p_{s+d}\}$ is focused on $V(P)$, we only need to consider the sets 
\[N_G[p_s]\cap N_{G_{l'}}[p_{s+d}]^c\cap V(P)\text{ and }N_G[p_{s+d}] \cap N_{G_{l'}}[p_s]^c\cap V(P).\]
It is straightforward to verify that the first set is empty and that the second set contains only the vertex $p_{i+d+1}$.  By Lemma~\ref{thm:rule1}, the force $\{p_s,p_{s+d}\}\forces{G}{G_{l'}} \{p_s,p_{s+d+1}\}$ is valid.  Inductively, we get the desired result.
\end{proof}

\begin{corollary}\label{cor:corofcor}
 Let $G$ be a supergraph of a graph $H$, which is obtained from $H$ by adding a path with vertices $v=v_0, v_1,\ldots,v_m$ to a vertex $v \in V(H)$.  If $A\in\S(G)$, $X\in\Szb(G^c)$, and $[A,X]=0$, then $X\in\Szb(G_{l+1}^c)$ with $G_{l+1} = G_l + \{\{v_i,v_j\}\colon 0 \leq i < j \leq m\}$.
\end{corollary}

While Lemma~\ref{thm:rule1} identifies an equation in $[A,X]=0$ with a single surviving term, the next two rules explore a set of equations with two surviving terms. 

\begin{lemma}\label{thm:rule2}
Let $G$ be a graph and $G_l$ a supergraph of $G$ of the same order.  Suppose $A\in\S(G)$, $X\in\Szb(G_l^c)$, and $[A,X]=0$.  If there exists a vertex $i$ such that $G_l^c[N_G(i)]$ contains a component that is an odd cycle $C$ and for each $j\in V(C)$ the pair $\{i,j\}$ is focused on $V(C)$ with respect to $G$ and $G_l$, then $X\in\Szb(G_{l+1}^c)$ with $G_{l+1}=G_l + E(C)$.
\end{lemma}
 
\begin{proof}
Suppose the odd cycle $C$ is of length $s$.  Label the vertices of $C$ as $j_1,j_2,\ldots ,j_s$. Since $\{i,j_m\}$ is focused on $V(C)$ for each $j_m\in V(C)$, we have 
\[N_{G}[j_m]\cap N_{G_l}[i]^c = N_{G}[j_m]\cap N_{G_l}[i]^c\cap V(C) = \emptyset \]
and
\[N_G[i]\cap N_{G_l}[j_m]^c= N_G[i]\cap N_{G_l}[j_m]^c\cap V(C) = \{j_{m+1},j_{m-1}\},\] 
where $j_{m+1}=j_1$ if $m=s$ and $j_{m-1}=j_s$.  Equation \eqref{ceneqn} is of the form:
\begin{align*}
   0&=[A,X]_{i,j_m}
   =a_{ij_{m-1}}x_{j_{m-1}j_m}+a_{ij_{m+1}}x_{j_{m+1}j_m},
\end{align*}
which implies 
\[\begin{bmatrix}
a_{ij_1} & a_{ij_3} & 0 & \cdots & 0 \\
0 & a_{ij_2} & a_{ij_4} & \ddots & \vdots \\
\vdots & \ddots & \ddots & \ddots & 0 \\
0 & \cdots & 0 & a_{ij_{s-1}} & a_{ij_1} \\
a_{ij_2} & 0 & \cdots & 0 & a_{ij_s} \\
\end{bmatrix}
\begin{bmatrix}
x_{j_1j_2}\\
x_{j_2j_3}\\
\vdots \\
 \\
x_{j_sj_1}
\end{bmatrix}=\begin{bmatrix}
0\\
0\\
\vdots \\
 \\ 
0
\end{bmatrix}.\]
Since $s$ is odd, the matrix on the left hand side has the determinant equal to
\[2\prod_{m=1}^s  a_{ij_m}\neq 0.\]
In particular, this implies $x_{j_mj_{m+1}}=0$ for $m=1,\ldots ,s$, proving that  $X\in\Szb(G_{l+1}^c)$ with $G_{l+1} = G + E(C)$.
\end{proof}

\begin{definition}
When the conditions of Lemma~\ref{thm:rule2} are satisfied and we update $G_l$ to $G_{l+1}$, and we say \emph{the vertex $i$ forces the cycle $C$}. We denote this by $$i\forces{G}{G_l} C.$$ 
\end{definition}

\begin{example}
For $n\geq 4$ and $G=K_{n-3}\vee(3K_1)=K_n - C_3$, let $i$ be any vertex in $K_{n-3}$ and $C$ be the  3-cycle in $G^c.$ By Lemma~\ref{thm:rule2} we have $i \forces{G}{G} C$, hence $G_1=K_n$, and so
$ K_n - C_3 \in \GG$.
\end{example}

Recall that $G^r$ and $G^{(r)}$ are the $r$-th power and the $r$-th strong power of the graph $G$, respectively.

\begin{lemma}\label{thm:rule3}
Let $G$ be a graph and $G_l$ a supergraph of $G$ of the same order.  Suppose $A\in\S(G)$, $X\in\Szb(G_l^c)$, and $[A,X]=0$.  Let $Y_h$, $h \geq 1$, be the graph obtained from three copies of the path $P_h$ by adding a vertex that is joined to an endpoint to each of the three paths.  If $Y_h$ is an induced subgraph of $G$ such that  
\begin{enumerate}
\item $E(Y_h^{(h)})\subseteq E(G_l)$ and $E(Y_h^{(h+1)})\cap E(G_l)=\emptyset$, and 
\item all $\{u,v\}\in E(Y_h^{(h)})$ are focused on $V(Y_h)$ with respect to $G$ and $G_l$,
\end{enumerate}
then $X\in\Szb(G_{l+1}^c)$ with $G_{l+1}=G_l +E(Y^{(h+1)})$.\end{lemma}
\begin{proof}
Let $P_{(i)}$ be a path on $h$ vertices $v^i_1,\ldots ,v^i_h$ for $i=1,2,3$, and suppose $Y_h$ is constructed from $P_{(1)}\dunion P_{(2)}\dunion P_{(3)}$ by adding a new vertex $v_0$ joined to $v^1_1$, $v^2_1$, and $v^3_1$.  Let $\alpha$ be an ordered set of $E(Y_h^{(h)})$: 
\begin{align*}
\alpha=\{ &\{v_0,v^1_h\}, \{v^2_1, v^1_{h-1}\},\ldots, \{v_0,v^2_h,\},
         \{v^3_1,v^2_{h-1}\},\ldots, \{v_0,v^3_h\},\\
        & \{v^1_1v^3_{h-1}\},\ldots,\{v^1_{h-1}v^3_1\}\},
\end{align*}
and let us name the edges in $\alpha$ as $e_1,\ldots,e_{3h},$ respecting the order.  Similarly, let $\beta$ be an ordered set of $E(Y_h^{(h+1)})$ with the order 
\[\{\{v^2_1,v^1_h\},\ldots,\{v^2_h,v^1_1\},\{v^3_1,v^2_h\},\ldots,\{v^3_h,v^2_1\},\{v^1_1,v^3_h\},\ldots,\{v^1_h,v^3_1\}\},\]
and the edges $f_1,\ldots,f_{3h}$, again respecting the order.

This result is based on considering those equations from the system $[A,X]=0$ that are in positions indexed by $\alpha$. For $l=1,h+1,2h+1$, we have $e_l=\{v_0, v^q_h\}$ for $q=1,2,3$, respectively. The corresponding equation has at least two terms, namely
\[a_{v_0v^{q+1}_1}x_{f_l}+a_{v_0v^{q-1}_1}x_{f_{l-1}},\] 
and  by Condition 2, those are the only two terms. All other elements in $\alpha$ are of the form $\{v^q_b, v^{q-1}_{h-b}\}$. The corresponding equation again has exactly two terms: 
\[a_{v^q_bv^q_{b+1}}x_{f_l}-a_{v^{q-1}_{h-b}v^{q-1}_{h-b+1}}x_{f_{l-1}}.\]

In summary, we can record the equations we have identified above as $M\bx=\bzero,$ where $M$ is 
\[
\begin{bmatrix}
a_{v_0v^2_1} & & & & & & & & a_{v_0v^3_1}\\
-a_{v^1_{h-1}v^1_h} & a_{v^2_1v^2_2} & & & & & & & \\
 & \ddots & \ddots &  & & & & & \\
 & & a_{v_0v^1_1} & a_{v_0v^3_1} & & & & & \\ 
 & & & -a_{v^2_{h-1}v^2_h} & a_{v^3_1v^3_2} & & & & \\
 & & & & \ddots & \ddots & & & \\ 
 & & & & & a_{v_0v^2_1} & a_{v_0v^1_1} & & \\ 
 & & & & & & -a_{v^3_{h-1}v^3_{h}} & a_{v^1_{1}v^1_2} & \\ 
 & & & & & & & \ddots & \ddots \\ 
\end{bmatrix},
\]
 and $\bx=(x_{f_1},\ldots,x_{f_{3h}})\trans $. Since $\det(M) \neq 0$ for any choice of $A \in \S(G)$, it follows that $x_{f_1}=\cdots =x_{f_{3h}}=0$.
\end{proof}

\begin{definition}
When the conditions of Lemma~\ref{thm:rule3} are satisfied we update $G_l$ to $G_{l+1}$, we say that $Y^{(h)}$ \emph{forces} $Y^{(h+1)}$ and we denote this by $$Y^{(h)}\forces{G}{G_l} Y^{(h+1)}.$$ 
\end{definition}

Applications of Lemma~\ref{thm:rule3} will be discussed in Section~\ref{sec:trees}.

\begin{theorem}\label{sspseq}
 If $G$ has a sequence: 
   \[G=G_0 \subset G_1 \subset \ldots \subset G_r=K_n,\]
such that for $i=0,\ldots, r-1$ each $G_{l+1}$ is constructed from $G_{l}$ by adding some edges by following one (or more) of the rules covered in Lemmas \ref{thm:rule1}, \ref{thm:rule2}, \ref{thm:rule3}, then $G \in \GG$.
\end{theorem}

While Lemma \ref{thm:rule1} is relatively straightforward to apply, Lemmas \ref{thm:rule2} and \ref{thm:rule3} are more involved, and are presented here with specific applications in mind. Moreover, it would be possible to develop other rules in this spirit, but they would be even harder to describe from the graph structure. We offer some examples that expose the limitations of the rules we developed above.

\begin{example}
Let $G=Y_2^{(2)}$ as shown below. 
\begin{center}
   \begin{tikzpicture}
\node [label={[label distance=-1.5pt]above:$1$}] (1) at (0:0) {};
\node [label={above:$4$}] (4) at (60:1) {};
\node [label={above:$7$}] (7) at (60:2) {};
\node [label={right:$3$}] (3) at (300:1) {};
\node [label={right:$6$}] (6) at (300:2) {};
\node [label={below:$2$}] (2) at (180:1) {};
\node [label={below:$5$}] (5) at (180:2) {};
\draw (5) -- (2) -- (1);
\draw (7) -- (4) -- (1);
\draw (6) -- (3) -- (1);
\draw (2) -- (4) -- (3) -- (2);
\draw (5) to[bend left] (1);
\draw (7) to[bend left] (1);
\draw (6) to[bend left] (1);
    \end{tikzpicture} 
\end{center}
Let $\alpha=\{\{1,2\},\{ 2,3\},\{1,3\},\{3,4\},\{1,4\},\{2,4\}\}$ be an ordered subset of $E(Y_2^{(2)})$  and 
 $\beta=\{\{2,6\}, \{3,5\},\{3,7\},\{4,6\},\{4,5\},\{2,7\}\}$ be an ordered subset of $E((Y_2^{(2)}) ^c)$.
The set of equations $[A,X]_{i,j}=0$ with $\{i,j\}\in\alpha$ is equivalent to $M\bx=0$ with 
\[M=
\begin{bmatrix}
-a_{1,6} & & & & &  -a_{1,7}\\
a_{3,6} & -a_{2,5} & &  & & \\
 & -a_{1,5} & -a_{1,7} & & & \\ 
 & & a_{4,7} & -a_{3,6} & & \\
 & & & -a_{1,6} & -a_{1,5} &  \\ 
 & & & & -a_{2,5} & a_{4,7}\\ 
\end{bmatrix}
\text{ and }
\bx=\begin{bmatrix}
x_{2,6} \\ x_{3,5} \\ x_{3,7} \\ x_{4,6} \\ x_{4,5} \\ x_{2,7}
\end{bmatrix}.\]
Since $\det(M)=-2 a_{1,6} a_{2,5} a_{1,7} a_{3,6} a_{1,5} a_{4,7}$ is always nonzero,
it follows that $\bx = 0$. (Note that this example is not covered by Lemma~\ref{thm:rule3}.)
\end{example}

The simplest example of a graph in $\GG$ that is not covered by any of the rules above that we could find is given below. 

\begin{example}\label{smallgraph}
 Let $G$ be a $C_4$ with an additional edge, as shown below. 
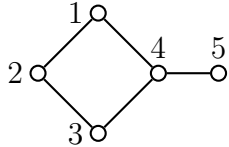
\begin{center} 
\begin{tikzpicture}
	 \node[label={above:$4$}] (1) at (0:1) {};
	 \node[label={left:$1$}] (2) at (90:1) {};
	 \node[label={left:$2$}] (3) at (180:1) {};
	 \node[label={left:$3$}] (4) at (270:1) {};
	 \node[label={above:$5$}] (5) at (0:2) {};
	 \draw (1) -- (2) -- (3) -- (4) -- (1) -- (5);
\end{tikzpicture}
\end{center}
   Let $A \in \S(G)$ and $X$ be a symmetric $5 \times 5$ matrix satisfying  $A \circ X=I_5 \circ X=0$ 
   and $[A,X]=0$. To prove that $G \in \GG$ we consider two subsets of equations of the linear system $[A,X]=0$, both involving indeterminates $x_{\gamma_j}$ where $$\gamma_j \in \gamma=\{\{1,3\},\{ 1,5\},\{2,4\},\{2,5\},\{3,5\}\} \subseteq E(G ^c).$$ In particular, we define $\bx=(x_{1,3},x_{1,5},x_{2,4},x_{2,5},x_{3,5})\trans $.
 
 The first subset of equations is in $[A,X]$ indexed by positions  $$\alpha=\{\{1,4\},\{ 1,2\},\{2,3\},\{1,5\},\{2,5\}\}  \subseteq E(G),$$ 
and has the coefficient matrix:
\[M=
\begin{bmatrix}
-a_{3,4} & -a_{4,5} & a_{1,2} & 0&  0\\
-a_{2,3} & 0 &a_{1,4} & 0 & 0\\
 a_{1,2}& 0& -a_{3,4} & 0 &0 & \\ 
 0& a_{1,1}-a_{5,5}& 0 & a_{1,2} & 0 \\
 0& a_{1,2}& -a_{4,5}&a_{2,2}-a_{5,5} & a_{2,3}\\ 
\end{bmatrix},
\] while the second subset of equations is in $[A,X]$ indexed by positions 
$$\beta=\{\{1,4\}, \{1,5\},\{2,5\},\{4,5\},\{3,4\}\} \subseteq E(G)$$
and has the coefficient matrix: 
\[N=
\begin{bmatrix}
-a_{3,4} & -a_{4,5} & a_{1,2} & 0&  0\\
 0& a_{1,1}-a_{5,5}& 0 & a_{1,2} & 0 \\
 0 & a_{1,2}& -a_{4,5}&a_{2,2}-a_{5,5} & a_{2,3}\\ 
 0&a_{1,4} & 0  & 0 &a_{3,4}\\
 -a_{1,4}& 0& a_{2,3} & 0 &-a_{4,5}\\ 
\end{bmatrix}.
\]
%

We need to consider the systems of equations $M\bx=0$, and $N\bx=0$ simultaneously. Note that $\det(M)=a_{1,2} a_{2,3} a_{4,5} (a_{2,3} a_{3,4}- a_{1,2}a_{1,4})$ allows us to conclude $\bx=0$ unless $a_{2,3} a_{3,4}- a_{1,2}a_{1,4}=0$. If $a_{2,3} a_{3,4}- a_{1,2}a_{1,4} =0$, then $a_{1,4}=\frac{a_{2,3}  a_{3,4}}{a_{1,2}}$, and $\det N$ simplifies to $\det N=2a_{1,2}a_{1,4}a_{3,4}a_{4,5}^2 \ne 0$, again implying $\bx=0$. This proves that any $A \in \S(G)$ has the SSP and therefore $G \in \GG$.
\end{example}


\section{Trees}
\label{sec:trees}

In this section we characterise the trees in $\GG$.  
As a consequence of Corollary~\ref{rem:tnotinGG}, a tree in $\GG$ is either a path or a tree with a unique vertex of degree three.  First, we show that every path is contained in $\GG$.

\begin{lemma}\label{ex:PnP}
  $P_n \in \GG$ for all $n \geq 1$.
\end{lemma}

\begin{proof} 
Let $G=G_0=P_n$.  That is, $V(G)=\{1,2,\ldots,n\}$ and $$E(G)=\{\{i,i+1\}; 1 \leq i \leq n-1\}.$$ 
First we notice that $\{1,2\}\forces{G}{G_0} \{1,3\}$, and we define $G_1=G+\{1,3\}$. Now we have $\{2,3\}\forces{G}{G_1} \{2,4\},$ and we define $G_2=G_1+\{2,4\}$. We proceed in this way, altogether making the following steps:  
\[\{i,i+1\}\forces{G}{G_{i-1}} \{i,i+2\}  \text{ for }i=1,2, \ldots, n-2.\]
The resulting graph $G_{n-2}$ is equal to $G^2$. 

Next we add the edges that connect points at distance three in $G$.  
 We can make the following sequence: 
   \[\{i,i+2\}\forces{G}{G_{n-2+i}} \{i,i+3\}  \text{ for }i=1,2, \ldots, n-2,\]
  resulting in the graph  $G_{2n-5}=G^3$.  We continue by adding edges that connect points at the distance four in $G$, then the ones the distance five in $G$, and so on.  This process completes with
  \[\{1,n-1\}\forces{G}{G_{r-1}} \{1,n\},\]
   and $G_{r}=G^{n-1}=K_n$, where $r= \binom{n-2}{2}$.
  We conclude that $X\in\Szb(K_n^c)$ has to be the zero matrix and therefore  $P_n\in\GG$. 
\end{proof}

The main result in this section states that every tree with a unique vertex of degree three is in $\GG$, and thus characterises the trees in $\GG$.
To illustrate the steps needed, and to help the reader follow the proof in the general case, we fist offer an example.  By $P_{u,v}$ we will denote the unique path from vertex $u$ to vertex $v$ in the given tree.

\begin{example}\label{ex:trees}
Let $G$ be the tree obtained by identifying leaves of $P_3$, $P_4$, and $P_5$ as shown below.  Let $n=|V(G)|$.  By applying the tools from Section \ref{Sec3} we will show that $G$ is an SSP graph. 

 
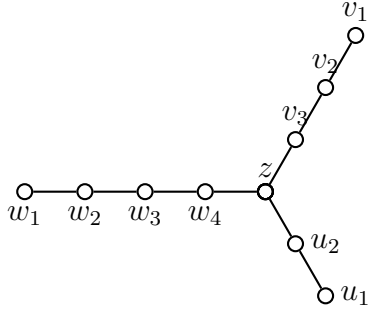
\begin{center}
   \begin{tikzpicture}
\node[label={above:$z$}] (z)   at (0,0) {}; 

\node[label={right:$u_1$}] (u1)   at (300:2) {}; 
\node[label={right:$u_2$}] (u2)   at (300:1) {}; 
\node (u3) at (0,0) {};
\draw (u1) -- (u2) -- (u3);

\node[label={above:$v_1$}] (v1)   at (60:3) {}; 
\node[label={above:$v_2$}] (v2)   at (60:2) {}; 
\node[label={above:$v_3$}] (v3)   at (60:1) {}; 
\node (v4) at (0,0) {};
\draw (v1) -- (v2) -- (v3) -- (v4);

\node[label={below:$w_1$}] (w1)   at (180:4) {}; 
\node[label={below:$w_2$}] (w2)   at (180:3) {}; 
\node[label={below:$w_3$}] (w3)   at (180:2) {}; 
\node[label={below:$w_4$}] (w4)   at (180:1) {}; 
\node (w5) at (0,0) {};
\draw (w1) -- (w2) -- (w3) -- (w4) -- (w5);
\end{tikzpicture} 
  \end{center}  
   Let us label the vertices as indicated in the picture, with the convention that $z=u_3=v_4=w_5$. Let $U=\{u_i\}_{i=1}^3$, $V=\{v_i\}_{i=1}^4$, and $W=\{w_i\}_{i=1}^5$.  We will establish that $G \in \GG$ in four phases. 
    
 Let $A \in \S(G)$ and $X$ be a symmetric $10 \times 10$ matrix satisfying $$A \circ X=I_{10} \circ X=0\text{ and }[A,X]=0.$$
Initially we know that $X_{ij}=0$, if $\{i,j\} \in E(G)$ or $i=j$. Step by step we will establish other zero entries in $X$. To keep track of this process we will define a sequence of graphs $G_i$ according to the set up in Section 3. 
 
{\bf Phase 1.}  We apply Corollary~\ref{cor:corofcor}  to the paths $P_{u_1,z}$, $P_{v_1,z}$ and $P_{w_1,z}$, respectively, and
 conclude that $X\in\Szb(G_1)$, where $G_1$ is obtained from $G$ by adding to it all possible edges with both vertices either in $U$, in $V$, or in $W$. 
   
{\bf Phase 2.}  Let $C$ be the $3$-cycle $G_1^c[N_G(z)]$.  Since $N_{G_1}[z]^c=\emptyset$, we have $z \rightarrow C$ by Lemma~\ref{thm:rule2}.  This leads to  $X\in\Szb(G_2)$ with $G_2 = G_1 + C$. 

{\bf Phase 3.} 
Let $Y_2$ be the induced subgraph of $G$ on vertices whose distance to $z$ is at most $2$.  Then $G_2[V(Y_2)]=Y_2^2$. 
 Next we check that every $\{x,y\}\in E(Y_2^{(2)})$ is focused on $V(Y_2)$.  Suppose $\{u_i,v_j\}\in E(Y_2^{(2)})$.
  If neither $u_i$ nor $v_j$ is $z$, then $N_G[u_i]\subseteq V(Y_2)$, $N_G[v_j]\subseteq V(Y_2)$, and $\{u_i,v_j\}$ is focused on $V(Y_2)$.  If $v_j=z$, then $N_G[u_i]$ contains a vertex $u_{i-1}$ outside $V(Y_2)$, but since $N_{G_2}[z]$ also contains $u_{i-1}$, we have that $\{u_i,v_j\}$ is focused on $V(Y_2)$.  The case when $u_i=z$ and the cases for $\{v_i,w_j\}$ and $\{w_i,u_j\}$ are similar.  Thus, Lemma~\ref{thm:rule3} shows that $X\in\Szb(G_{3})$ with $G_{3}=G_2+Y_2^3$.  

\newcommand{\exscale}{0.8}
 \begin{center}
 \begin{figure}[htb!]
 \centering
 \subfigure	[$G_1$ after Phase 1]{
    \begin{tikzpicture}[scale=\exscale,transform shape]
\node[label={above:$z$}] (z)   at (0,0) {}; 

\node[label={right:$u_1$}] (u1)   at (300:2) {}; 
\node[label={right:$u_2$}] (u2)   at (300:1) {}; 
\node (u3) at (0,0) {};
\draw (u1) -- (u2) -- (u3);

\node[label={above:$v_1$}] (v1)   at (60:3) {}; 
\node[label={above:$v_2$}] (v2)   at (60:2) {}; 
\node[label={above:$v_3$}] (v3)   at (60:1) {}; 
\node (v4) at (0,0) {};
\draw (v1) -- (v2) -- (v3) -- (v4);

\node[label={below:$w_1$}] (w1)   at (180:4) {}; 
\node[label={below:$w_2$}] (w2)   at (180:3) {}; 
\node[label={below:$w_3$}] (w3)   at (180:2) {}; 
\node[label={below:$w_4$}] (w4)   at (180:1) {}; 
\node (w5) at (0,0) {};
\draw (w1) -- (w2) -- (w3) -- (w4) -- (w5);

\begin{scope}[green]
\foreach \letter/\length in {u/3, v/4, w/5}{
    \pgfmathsetmacro{\lengthm}{int(\length-1)}
    \foreach \i in {1,...,\lengthm}{
        \pgfmathsetmacro{\ip}{int(\i+1)}
        \foreach \j in {\ip,...,\length}{
            \draw (\letter\i) to[bend left] (\letter\j);
        }
    }
}
\end{scope}
    \end{tikzpicture}\label{Ph1}}
 \subfigure	[$G_2$ after Phase 2]{
 \begin{tikzpicture}[scale=\exscale,transform shape]
\node[label={above:$z$}] (z)   at (0,0) {}; 

\node[label={right:$u_1$}] (u1)   at (300:2) {}; 
\node[label={right:$u_2$}] (u2)   at (300:1) {}; 
\node (u3) at (0,0) {};
\draw (u1) -- (u2) -- (u3);

\node[label={above:$v_1$}] (v1)   at (60:3) {}; 
\node[label={above:$v_2$}] (v2)   at (60:2) {}; 
\node[label={above:$v_3$}] (v3)   at (60:1) {}; 
\node (v4) at (0,0) {};
\draw (v1) -- (v2) -- (v3) -- (v4);

\node[label={below:$w_1$}] (w1)   at (180:4) {}; 
\node[label={below:$w_2$}] (w2)   at (180:3) {}; 
\node[label={below:$w_3$}] (w3)   at (180:2) {}; 
\node[label={below:$w_4$}] (w4)   at (180:1) {}; 
\node (w5) at (0,0) {};
\draw (w1) -- (w2) -- (w3) -- (w4) -- (w5);
 
\begin{scope}[black!50]
\foreach \letter/\length in {u/3, v/4, w/5}{
    \pgfmathsetmacro{\lengthm}{int(\length-1)}
    \foreach \i in {1,...,\lengthm}{
        \pgfmathsetmacro{\ip}{int(\i+1)}
        \foreach \j in {\ip,...,\length}{
            \draw (\letter\i) to[bend left] (\letter\j);
        }
    }
}
\end{scope}

\begin{scope}[green]
\draw (u2) -- (v3) -- (w4) -- (u2);
\end{scope}
    \end{tikzpicture}   \label{Ph2}} 
   	\subfigure[$G_3$ after Phase 3]{ 
\begin{tikzpicture}[scale=\exscale,transform shape]
\node[label={above:$z$}] (z)   at (0,0) {}; 

\node[label={right:$u_1$}] (u1)   at (300:2) {}; 
\node[label={[label distance=5pt]right:$u_2$}] (u2)   at (300:1) {}; 
\node (u3) at (0,0) {};
\draw (u1) -- (u2) -- (u3);

\node[label={above:$v_1$}] (v1)   at (60:3) {}; 
\node[label={above:$v_2$}] (v2)   at (60:2) {}; 
\node[label={[xshift=-10pt,yshift=3pt]above:$v_3$}] (v3)   at (60:1) {}; 
\node (v4) at (0,0) {};
\draw (v1) -- (v2) -- (v3) -- (v4);

\node[label={below:$w_1$}] (w1)   at (180:4) {}; 
\node[label={below:$w_2$}] (w2)   at (180:3) {}; 
\node[label={below:$w_3$}] (w3)   at (180:2) {}; 
\node[label={[label distance=6pt]below:$w_4$}] (w4)   at (180:1) {}; 
\node (w5) at (0,0) {};
\draw (w1) -- (w2) -- (w3) -- (w4) -- (w5);

\begin{scope}[black!50]
\foreach \letter/\length in {u/3, v/4, w/5}{
    \pgfmathsetmacro{\lengthm}{int(\length-1)}
    \foreach \i in {1,...,\lengthm}{
        \pgfmathsetmacro{\ip}{int(\i+1)}
        \foreach \j in {\ip,...,\length}{
            \draw (\letter\i) to[bend left] (\letter\j);
        }
    }
}
\end{scope}

\begin{scope}[black!50]
\draw (u2) -- (v3) -- (w4) -- (u2);
\end{scope}

\begin{scope}[green]
\draw (u1) -- (v3) -- (w3) -- (u2) -- (v2) -- (w4) -- (u1);
\end{scope}
    \end{tikzpicture} \label{Ph3}} 
    \caption{Sequence of graphs $G_1 \subset G_2 \subset G_3$. The new edges 
    in the supergraphs are colored by green.}\label{P2P3P3}
\end{figure}
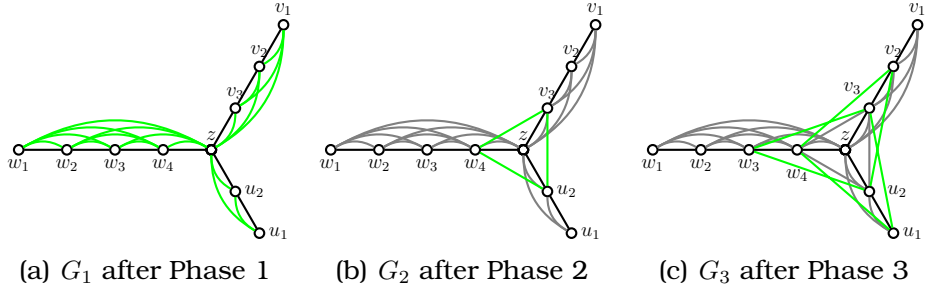
\end{center}

 \begin{center}
 \begin{figure}[htb!]
 \centering
\subfigure[Phase 4a]{
\begin{tikzpicture}[scale=\exscale,transform shape]
\node[label={above:$z$}] (z)   at (0,0) {}; 

\node[label={right:$u_1$}] (u1)   at (300:2) {}; 
\node[label={[label distance=5pt]right:$u_2$}] (u2)   at (300:1) {}; 
\node (u3) at (0,0) {};
\draw (u1) -- (u2) -- (u3);

\node[label={above:$v_1$}] (v1)   at (60:3) {}; 
\node[label={above:$v_2$}] (v2)   at (60:2) {}; 
\node[label={[xshift=-10pt,yshift=3pt]above:$v_3$}] (v3)   at (60:1) {}; 
\node (v4) at (0,0) {};
\draw (v1) -- (v2) -- (v3) -- (v4);

\node[label={below:$w_1$}] (w1)   at (180:4) {}; 
\node[label={below:$w_2$}] (w2)   at (180:3) {}; 
\node[label={[xshift=-2pt]below:$w_3$}] (w3)   at (180:2) {}; 
\node[label={[label distance=6pt]below:$w_4$}] (w4)   at (180:1) {}; 
\node (w5) at (0,0) {};
\draw (w1) -- (w2) -- (w3) -- (w4) -- (w5);

\begin{scope}[black!50]
\foreach \letter/\length in {u/3, v/4, w/5}{
    \pgfmathsetmacro{\lengthm}{int(\length-1)}
    \foreach \i in {1,...,\lengthm}{
        \pgfmathsetmacro{\ip}{int(\i+1)}
        \foreach \j in {\ip,...,\length}{
            \draw (\letter\i) to[bend left] (\letter\j);
        }
    }
}
\end{scope}

\begin{scope}[black!50]
\draw (u2) -- (v3) -- (w4) -- (u2);
\end{scope}

\begin{scope}[black!50]
\draw (u1) -- (v3) -- (w3) -- (u2) -- (v2) -- (w4) -- (u1);
\end{scope}

\begin{scope}[green]
\foreach \u/\v in {u1/v2,u2/v1}{
  \draw (\u) to[bend right] (\v);
}

\foreach \u/\w in {u1/w3,u2/w2}{
  \draw (\u) to[bend left] (\w);
}

\foreach \v/\w in {v1/w4,v2/w3,v3/w2}{
  \draw (\v) to[bend right] (\w);
}

\end{scope}

\begin{scope}
\node[rectangle,draw=none,below,align=center] at (-1.25,-2) {
    $\{w_4,u_1\} \rightarrow \{u_1,w_3\}$\\
    $\{w_3,u_2\} \rightarrow \{u_2,w_2\}$\\
    $\{v_3,u_1\} \rightarrow \{u_1,v_2\}$\\
    $\{v_2,u_2\} \rightarrow \{u_2,v_1\}$\\
    $\{w_5,v_1\} \rightarrow \{v_1,w_4\}$\\
    $\{w_4,v_2\} \rightarrow \{v_2,w_3\}$\\
    $\{w_3,v_3\} \rightarrow \{v_3,w_2\}$\\
};
\node[draw=none,inner sep=0] at (0,-6) {};
\end{scope}
\end{tikzpicture}\label{4a}} 
\subfigure[Phase 4b]{   
\begin{tikzpicture}[scale=\exscale,transform shape]
\node[label={above:$z$}] (z)   at (0,0) {}; 

\node[label={right:$u_1$}] (u1)   at (300:2) {}; 
\node[label={[label distance=5pt]right:$u_2$}] (u2)   at (300:1) {}; 
\node (u3) at (0,0) {};
\draw (u1) -- (u2) -- (u3);

\node[label={above:$v_1$}] (v1)   at (60:3) {}; 
\node[label={above:$v_2$}] (v2)   at (60:2) {}; 
\node[label={[xshift=-10pt,yshift=3pt]above:$v_3$}] (v3)   at (60:1) {}; 
\node (v4) at (0,0) {};
\draw (v1) -- (v2) -- (v3) -- (v4);

\node[label={below:$w_1$}] (w1)   at (180:4) {}; 
\node[label={below:$w_2$}] (w2)   at (180:3) {}; 
\node[label={[xshift=-2pt]below:$w_3$}] (w3)   at (180:2) {}; 
\node[label={[label distance=6pt]below:$w_4$}] (w4)   at (180:1) {}; 
\node (w5) at (0,0) {};
\draw (w1) -- (w2) -- (w3) -- (w4) -- (w5);

\begin{scope}[black!50]
\foreach \letter/\length in {u/3, v/4, w/5}{
    \pgfmathsetmacro{\lengthm}{int(\length-1)}
    \foreach \i in {1,...,\lengthm}{
        \pgfmathsetmacro{\ip}{int(\i+1)}
        \foreach \j in {\ip,...,\length}{
            \draw (\letter\i) to[bend left] (\letter\j);
        }
    }
}
\end{scope}

\begin{scope}[black!50]
\draw (u2) -- (v3) -- (w4) -- (u2);
\end{scope}

\begin{scope}[black!50]
\draw (u1) -- (v3) -- (w3) -- (u2) -- (v2) -- (w4) -- (u1);
\end{scope}

\begin{scope}[black!50]
\foreach \u/\v in {u1/v2,u2/v1}{
  \draw (\u) to[bend right] (\v);
}

\foreach \u/\w in {u1/w3,u2/w2}{
  \draw (\u) to[bend left] (\w);
}

\foreach \v/\w in {v1/w4,v2/w3,v3/w2}{
  \draw (\v) to[bend right] (\w);
}

\end{scope}

\begin{scope}[green]
\foreach \u/\v in {u1/v1}{
  \draw (\u) to[bend right] (\v);
}

\foreach \u/\w in {u1/w2,u2/w1}{
  \draw (\u) to[bend left] (\w);
}

\foreach \v/\w in {v1/w3,v2/w2,v3/w1}{
  \draw (\v) to[bend right] (\w);
}
\end{scope}

\begin{scope}
\node[rectangle,draw=none,below,align=center] at (-1.25,-2) {
    $\{w_3,u_1\} \rightarrow \{u_1,w_2\}$\\
    $\{w_2,u_2\} \rightarrow \{u_2,w_1\}$\\
    $\{v_2,u_1\} \rightarrow \{u_1,v_1\}$\\
    $\{w_4,v_1\} \rightarrow \{v_1,w_3\}$\\
    $\{w_3,v_2\} \rightarrow \{v_2,w_2\}$\\
    $\{w_2,v_3\} \rightarrow \{v_3,w_1\}$\\
};
\end{scope}
\node[draw=none,inner sep=0] at (0,-6) {};
\end{tikzpicture} \label{4b}}
\subfigure[Phase 4c]{
\begin{tikzpicture}[scale=\exscale,transform shape]
\node[label={above:$z$}] (z)   at (0,0) {}; 

\node[label={right:$u_1$}] (u1)   at (300:2) {}; 
\node[label={[label distance=5pt]right:$u_2$}] (u2)   at (300:1) {}; 
\node (u3) at (0,0) {};
\draw (u1) -- (u2) -- (u3);

\node[label={above:$v_1$}] (v1)   at (60:3) {}; 
\node[label={above:$v_2$}] (v2)   at (60:2) {}; 
\node[label={[xshift=-10pt,yshift=3pt]above:$v_3$}] (v3)   at (60:1) {}; 
\node (v4) at (0,0) {};
\draw (v1) -- (v2) -- (v3) -- (v4);

\node[label={below:$w_1$}] (w1)   at (180:4) {}; 
\node[label={below:$w_2$}] (w2)   at (180:3) {}; 
\node[label={[xshift=-2pt]below:$w_3$}] (w3)   at (180:2) {}; 
\node[label={[label distance=6pt]below:$w_4$}] (w4)   at (180:1) {}; 
\node (w5) at (0,0) {};
\draw (w1) -- (w2) -- (w3) -- (w4) -- (w5);

\begin{scope}[black!50]
\foreach \letter/\length in {u/3, v/4, w/5}{
    \pgfmathsetmacro{\lengthm}{int(\length-1)}
    \foreach \i in {1,...,\lengthm}{
        \pgfmathsetmacro{\ip}{int(\i+1)}
        \foreach \j in {\ip,...,\length}{
            \draw (\letter\i) to[bend left] (\letter\j);
        }
    }
}
\end{scope}

\begin{scope}[black!50]
\draw (u2) -- (v3) -- (w4) -- (u2);
\end{scope}

\begin{scope}[black!50]
\draw (u1) -- (v3) -- (w3) -- (u2) -- (v2) -- (w4) -- (u1);
\end{scope}

\begin{scope}[black!50]
\foreach \u/\v in {u1/v2,u2/v1}{
  \draw (\u) to[bend right] (\v);
}

\foreach \u/\w in {u1/w3,u2/w2}{
  \draw (\u) to[bend left] (\w);
}

\foreach \v/\w in {v1/w4,v2/w3,v3/w2}{
  \draw (\v) to[bend right] (\w);
}

\end{scope}

\begin{scope}[black!50]
\foreach \u/\v in {u1/v1}{
  \draw (\u) to[bend right] (\v);
}

\foreach \u/\w in {u1/w2,u2/w1}{
  \draw (\u) to[bend left] (\w);
}

\foreach \v/\w in {v1/w3,v2/w2,v3/w1}{
  \draw (\v) to[bend right] (\w);
}
\end{scope}

\begin{scope}[green]
\foreach \u/\w in {u1/w1}{
  \draw (\u) to[bend left] (\w);
}

\foreach \v/\w in {v1/w2,v2/w1}{
  \draw (\v) to[bend right] (\w);
}
\end{scope}

\begin{scope}
\node[rectangle,draw=none,below,align=center] at (-1.25,-2) {
    $\{w_2,u_1\} \rightarrow \{u_1,w_1\}$\\
    $\{w_3,v_1\} \rightarrow \{v_1,w_2\}$\\
    $\{w_2,v_2\} \rightarrow \{v_2,w_1\}$\\
};
\end{scope}
\node[draw=none,inner sep=0] at (0,-6) {}; 
    \end{tikzpicture} \label{4c}}
\subfigure[Phase 4d]{
\begin{tikzpicture}[scale=\exscale,transform shape]
\node[label={above:$z$}] (z)   at (0,0) {}; 

\node[label={right:$u_1$}] (u1)   at (300:2) {}; 
\node[label={[label distance=5pt]right:$u_2$}] (u2)   at (300:1) {}; 
\node (u3) at (0,0) {};
\draw (u1) -- (u2) -- (u3);

\node[label={above:$v_1$}] (v1)   at (60:3) {}; 
\node[label={above:$v_2$}] (v2)   at (60:2) {}; 
\node[label={[xshift=-10pt,yshift=3pt]above:$v_3$}] (v3)   at (60:1) {}; 
\node (v4) at (0,0) {};
\draw (v1) -- (v2) -- (v3) -- (v4);

\node[label={below:$w_1$}] (w1)   at (180:4) {}; 
\node[label={below:$w_2$}] (w2)   at (180:3) {}; 
\node[label={[xshift=-2pt]below:$w_3$}] (w3)   at (180:2) {}; 
\node[label={[label distance=6pt]below:$w_4$}] (w4)   at (180:1) {}; 
\node (w5) at (0,0) {};
\draw (w1) -- (w2) -- (w3) -- (w4) -- (w5);

\begin{scope}[black!50]
\foreach \letter/\length in {u/3, v/4, w/5}{
    \pgfmathsetmacro{\lengthm}{int(\length-1)}
    \foreach \i in {1,...,\lengthm}{
        \pgfmathsetmacro{\ip}{int(\i+1)}
        \foreach \j in {\ip,...,\length}{
            \draw (\letter\i) to[bend left] (\letter\j);
        }
    }
}
\end{scope}

\begin{scope}[black!50]
\draw (u2) -- (v3) -- (w4) -- (u2);
\end{scope}

\begin{scope}[black!50]
\draw (u1) -- (v3) -- (w3) -- (u2) -- (v2) -- (w4) -- (u1);
\end{scope}

\begin{scope}[black!50]
\foreach \u/\v in {u1/v2,u2/v1}{
  \draw (\u) to[bend right] (\v);
}

\foreach \u/\w in {u1/w3,u2/w2}{
  \draw (\u) to[bend left] (\w);
}

\foreach \v/\w in {v1/w4,v2/w3,v3/w2}{
  \draw (\v) to[bend right] (\w);
}

\end{scope}

\begin{scope}[black!50]
\foreach \u/\v in {u1/v1}{
  \draw (\u) to[bend right] (\v);
}

\foreach \u/\w in {u1/w2,u2/w1}{
  \draw (\u) to[bend left] (\w);
}

\foreach \v/\w in {v1/w3,v2/w2,v3/w1}{
  \draw (\v) to[bend right] (\w);
}
\end{scope}

\begin{scope}[black!50]
\foreach \u/\w in {u1/w1}{
  \draw (\u) to[bend left] (\w);
}

\foreach \v/\w in {v1/w2,v2/w1}{
  \draw (\v) to[bend right] (\w);
}
\end{scope}

\begin{scope}[green]
\foreach \v/\w in {v1/w1}{
  \draw (\v) to[bend right] (\w);
}
\end{scope}

\begin{scope}
\node[rectangle,draw=none,below,align=center] at (-1.25,-2) {
    $\{w_2,v_1\} \rightarrow \{v_1,w_1\}$\\
};
\end{scope}
\node[draw=none,inner sep=0] at (0,-6) {}; 
    \end{tikzpicture}\label{4d}}
    \caption{Detailed sequence of graphs in Phase 4, all induced by Lemma \ref{thm:rule1}. The new edges are colored by green.}\label{ex:phase4}
\end{figure}
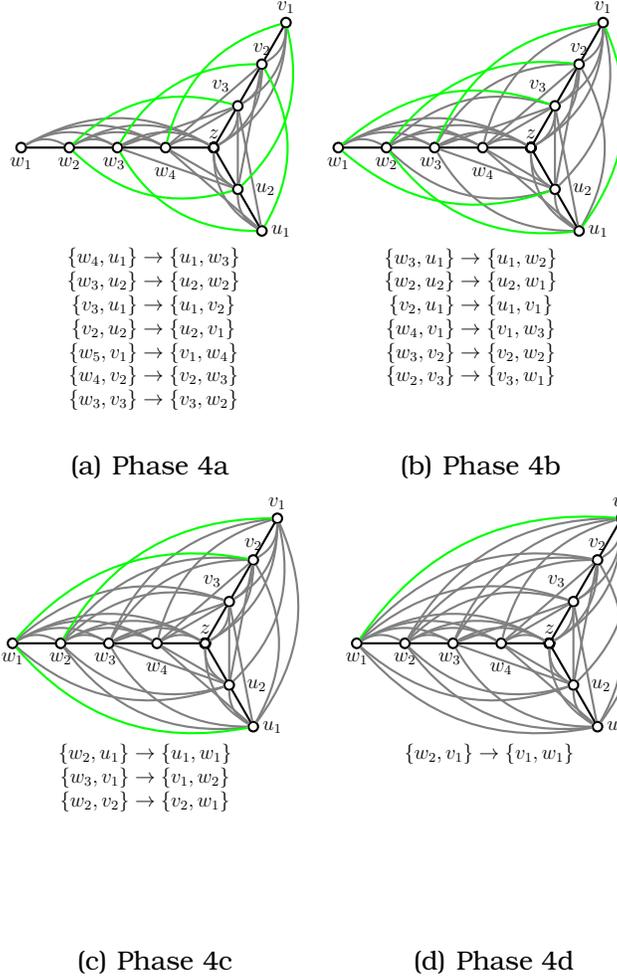
\end{center}

{\bf Phase 4.}  At this point, $E(G_3)$ contains all pairs of vertices at distance $3$ in $G$ and all pairs in $U$, $V$, $W$, respectively.  Next we will add all pairs at distance $k$  for $k=4,5,6,7$ inductively. Starting with $k=4$, we let $P=P_{u_1,v_1}$, and  we apply Corollary~\ref{cor:rule1cor} to $P$ with $d=3$.  This step adds all pairs at distance $4$ in $P_{u_1,v_1}$.  Next, let $P=P_{u_1,w_2}$.  (Note that the pair $\{z,w_2\}$ is not focused on $P$ if we let $P=P_{u_1,w_1}$, but this does not impair the process since $\{z,w_1\}$ is already included in $G_3$.)  Apply Corollary~\ref{cor:rule1cor} to $P$ with $d=3$ so that all pairs at distance $4$ on $P_{u_1,w_1}$ are included.  Finally, apply Corollary~\ref{cor:rule1cor} to $P=P_{v_1,w_1}$ to include all pairs at distance $4$ on $P_{v_1,w_1}$.  At this point, we may construct a graph $G_4$ by adding all pairs at distance $4$ in $G$ to $E(G_3)$ and conclude that $X\in\Szb(G_4)$.  Carrying out this process inductively with $k=5,6,7$, we reach the desired result that $X\in\Szb(K_n)$ and $X=0$.
\end{example}

\begin{theorem}\label{thm:trees} 
Let $G$ be a tree on $n$ vertices.  Then $G\in\GG$ if and only if it does not contain a vertex of degree at least four nor two vertices of degree three.
\end{theorem}

\begin{proof}
Corollary \ref{rem:tnotinGG} tells us that the conditions for a tree to be in $\GG$ listed in the theorem are necessary. We still need to prove that they are sufficient. 

If $G$ is a tree whose vertices all have degree less than three, then it is a path and by Lemma \ref{ex:PnP} we have $G \in \GG$. Suppose now $G$ has exactly one degree three vertex $z$.  We may assume that $G$ is obtained from three paths by identifying an endpoint of each of the paths.  More specifically, let $P^{(a)}$ be a path on vertices $\{u_1, \ldots, u_a\}$, $P^{(b)}$ the path on vertices $\{v_1, \ldots, v_b\}$, and $P^{(c)}$ the path on vertices $\{w_1, \ldots, w_c\}$.  Then $G$ is obtained from $P^{(a)}\dunion P^{(b)}\dunion P^{(c)}$ by identifying $u_a$, $v_b$, $w_c$.  For convenience, let $z=u_a=v_b=w_c$ in $G$. Without loss of generality, we may assume $a\leq b\leq c$.  Also, let $U=\{u_i\}_{i=1}^a$, $V=\{v_i\}_{i=1}^b$, and $W=\{w_i\}_{i=1}^c$.

 Let $A\in \S(G)$ and $X$ be a symmetric matrix satisfying $A \circ X=I \circ X=0$ and $[A,X]=0$. The rest of the proof is divided in four phases as in Example \ref{ex:trees}. 

{\bf Phase 1.}
Apply Corollary \ref{cor:corofcor} to the paths $P_{u_1,z}$, $P_{v_1,z}$ and $P_{w_1,z}$, respectively.
We conclude that $X\in\Szb(G_1)$, where $G_1$ is obtained from $G$ by adding all possible edges in $U$, $V$ and $W$, respectively. 

{\bf Phase 2.} Let $C$ be the $3$-cycle $G_1^c[N_G(z)]$.  Since $N_{G_1}[z]^c=\emptyset$, we have $z \rightarrow C$ by Lemma~\ref{thm:rule2}.  This leads to that $X\in\Szb(G_2)$ with $G_2 = G_1 + C$. (See also Figure \ref{Ph2}.)


{\bf Phase 3.}  Notice that for $h=2,\ldots,a$, $Y_h$ is an induced subgraph of $G$ on vertices whose distance with $z$ is at most $h$. 
Let us define a sequence of graphs $H_1=G_2$ and $H_h=G_2+Y_2^3+\ldots+Y_{h}^{h+1}$ for $h=2,\ldots,a-1$. For $h=2,\ldots,a-1$ we inductively apply Lemma~\ref{thm:rule3} as follows. 
First we note that $H_{h-1}[V(Y_h)]=Y_h^h$.  Next we check that every $\{x,y\}\in E(Y_h^{(h)})$ is focused on $V(Y_h)$.  Suppose $\{u_i,v_j\}\in E(Y_h^{(h)})$.  If none of $u_i$, $v_j$ is $z$, then $N_G[u_i]\subseteq V(Y_h)$, $N_G[v_j]\subseteq V(Y_h)$, and $\{u_i,v_j\}$ is focused on $V(Y_h)$.  If $v_j=z$, then $N_G[u_i]$ contains a vertex $u_{i-1}$ outside $V(Y_h)$, but since $N_{H_{h-1}}[z]$ also contains $u_{i-1}$, $\{u_i,v_j\}$ is focused on $V(Y_h)$.  The case when $u_i=z$ and the cases for $\{v_i,w_j\}$ and $\{w_i,u_j\}$ are similar.  Thus, Lemma~\ref{thm:rule3} shows that $X\in\Szb(H_{h+1})$.  Doing this process for $h=2,\ldots, a-1$, we conclude that $X\in\Szb(G_{3})$ by Lemma~\ref{thm:rule3}, where $G_{3}=G+Y_2^3+\cdots +Y_{a-1}^{a}$.

At the end of this phase we know that $X_{ij}=0$, if the distance between $i$ and $j$ is at most $a$, or both $i$ and $j$ are contained in one of $U$, $V$, $W$. (See also Figure \ref{Ph3}.)

{\bf Phase 4.}
At this point, $E(G_3)$ contains all pairs at distance $a$  in $G$ and all pairs in $U$, $V$, $W$, respectively.  Next we will add all pairs at distance $k$  for $k=a+1,\ldots,a+c-2$ inductively, starting with the case $k=a+1$.  Let $v$ be the vertex in $V$ that is furthest away from $z$ with $\dist(z,v)<k$. 
More precisely, $v=v_j$ with $j=\max\{1,c-(k-1)\}$.  Let $P=P_{u_1,v}$.  Apply Corollary~\ref{cor:rule1cor} to $P$ with $d=k-1$.  Then all pairs at distance $k$  on $P_{u_1,v}$ can be included.  
Next, let $w$ be the vertex in $V$ that is furthest away from $z$ with $\dist(z,v)<k$.  (That is, $w=w_j$ with $j=\max\{1,b-(k-1)\}$.)  Let $P=P_{u_1,w}$.  
Apply Corollary~\ref{cor:rule1cor} to $P$ with $d=k-1$ so that all pairs at distance $k$  on $P_{u_1,w_1}$ are included.  
Finally, apply Corollary~\ref{cor:rule1cor} to $P=P_{v_1,w_1}$ to include all pairs at distance $k$  on $P_{v_1,w_1}$.  In this way, we may construct a graph $G_{4}$ by adding all 
pairs at distance $k$  in $G$ to $E(G_3)$ and conclude that $X\in\Szb(G_4)$.  Doing this process inductively with $k=a+1,\ldots,a+c-2$, we reach the desired result that $X\in\Szb(K_n)$ and $X=0$.
\end{proof}

\section{Further Examples of Graphs in $\GG$}\label{SSPgraphs}

In this section we present further examples of graphs in $\GG$ that can be found with methods of this paper. The selection of examples is aiming to show different   features, and it is not intended to be exhaustive.


%

While the focus of this manuscript is to identify families of graphs in $\GG$, the techniques developed can be used more generally, i.e.\ to find matrices with the strong spectral property. For example,  once we have identified some matrices that have the SSP it is useful to know, how to construct other matrices with the SSP. A few observations of this type are collected in the following lemma. 

\begin{lemma}\label{blockSSP}
Let $A\in \S_n(\R)$ with the SSP and $B \in \S_m(\R)$. Moreover, let $C \in \R^{n \times m}$ have no zero entries.  Define
\[M=\begin{bmatrix}
   A & C\\
   C\trans  & B
 \end{bmatrix}.\]
Then $M$ has the SSP if one of the following conditions holds:  
\begin{enumerate}
   \item matrix $B$ has the SSP,
    \item $B \in \S(G)$, where $G^c$ is either a tree or an unicyclic graph with an odd cycle,
   \item $\rk C=m$ and $m \leq n$. \end{enumerate}
\end{lemma}

\begin{proof}
Let $$M=\begin{bmatrix}
   A & C\\
   C\trans  & B
 \end{bmatrix} \in \S_{n+m}(\R)$$ have the properties listed in the statement of the lemma. Let
$$X=Z \oplus W \in \S_{n+m}(\R),$$ where
 $A \circ Z= I_n \circ Z =0_n$,  $B \circ W= I_m \circ W =0_m$, and $[M,X]=0$. 
 Since $A$ has the SSP, we immediately get $Z=0$. Other conditions coming from  $[M,X]=0$  are
 $$   CW=0 \, \text{ and } \,  [B,W]=0.$$ 
 \begin{enumerate}
\item  If $B$ has the SSP, then $[B,W]=0$ implies $W=0$.
\item Let $W\in\Szb(G^c)$ satisfy $CW=0$. Since $C$ has no zero elements, the condition $CW=0$ implies that all the elements in $W$ corresponding to leaves in $G^c$ have to be zero. Let $G_1$ be the complement of the graph obtained from $G^c$ by deleting all the edges corresponding to leaves. We can now assume that $W\in\Szb(G_1^c)$. Since $G_1^c$ is again either a tree or an unicyclic graph with an odd cycle, we can repeat the argument we just made, to obtain $G_2$ by deleting edges corresponding to leaves in $G_1^c$, proving that $W \in \Szb(G_2^c)$. Repeated application of this argument proves that $W \in\Szb(G_l^c)$, where $G_l^c$ is either an empty graph, or an odd cycle with possibly some isolated vertices. If $G_l^c$ is an empty graph, then we have already shown that $W=0$. Otherwise, we apply Lemma \ref{thm:rule2} with $v\forces{G}{G_l} C$ to the odd cycle in $G_l^c$ and any vertex $v$ corresponding the first $n$ rows in the matrix, again showing that $G_{l+1}^c$ is an empty graph. 
\item  If $\rk C=m$ and $m \leq n$, then $CW=0$
 implies  $W=0$.
\end{enumerate}
\end{proof}


\begin{corollary}\label{vee forestC}
 Let $G\in \GG$ be a nonempty graph, then  $G\vee H \in \GG$ if $H$ satisfies one of the following conditions: 
 \begin{itemize} 
 \item $H \in \GG$ 
 \item $H^c$ is a union of trees and unicyclic graphs with an odd cycle. 
 \end{itemize}
\end{corollary}

Corollary \ref{vee forestC} can be used to identify several dense graphs in $\GG$. More generally, the condition $A \circ X=0$ forces $X$ to be sparse if $A$ is dense. In this case, one can expect the linear system $[A,X]=0$ to be easier to study, and matrices $A$ with the strong spectral property easier to identify.
The two examples that follow show that we cannot extend Corollary~\ref{vee forestC} to include unicyclic graphs with even cycles for $H$, or an empty graph for $G$. 

\begin{example}\label{ex:C4c}
 Choose $n \geq 4$ and let $$A=\left(
\begin{array}{cccc}
 1 & 0 & 1 & 0 \\
 0 & 1 & 0 & 1 \\
 1 & 0 & 1 & 0 \\
 0 & 1 & 0 & 1 \\
\end{array}
\right)\in \S(C_4^c),\, \,  X=\left(
\begin{array}{cccc}
 0 & 1 & 0 & 1 \\
 1 & 0 & 1 & 0 \\
 0 & 1 & 0 & 1 \\
 1 & 0 & 1 & 0 \\
\end{array}
\right) \in \S_4(\R),$$ and  $F=\npmatrix{
  J_{2,{n-4}}\\
  -J_{2,n-4} 
   } \in \R^{4 \times (n-4)}$, where $J_{k,l} \in \R^{k\times l}$ denotes the matrix full of ones.
    Now, let
 $$M=\npmatrix{
   A & F\\
  F & J_{n-4,n-4}
   } \in \S(K_n-C_4) \, \text{ and }\, Y=\npmatrix{
  X & 0\\
  0 & 0
   }\in \S_n(\R).
 $$
It straightforward to check that $Y \circ I_n=0$, $Y \circ M=0$ and $[Y,M]=0$,
implying that $K_n-C_4 \notin \GG$.
\end{example}

%
%
%
%

%

\begin{example}\label{K-t}
%
%

To show that $K_{2n}^{-n}\notin \GG$ we take: 
 $$A=\npmatrix{
 A_{11} & A_{12} & \ldots &A_{1n}\\
 A_{12}\trans  & A_{22} & \ldots &A_{2n}\\
 \vdots & \vdots & \ddots &\vdots\\
 A_{1n}\trans  & A_{2n}\trans  & \ldots &A_{nn}
 }\, \text{ and } \,  X=\npmatrix{
 Y & 0 & \ldots &0\\
 0 &Y & \ldots &0\\
 \vdots & \vdots & \ddots &\vdots\\
 0 & 0 & \ldots &Y
 }$$
  with
 $$A_{ij}=\begin{cases}
  \npmatrix{a_{ii} & 0\\
  0&a_{ii}} & i=j,\\
  \npmatrix{a_{ij} & b_{ij}\\
  b_{ij}&a_{ij}}, & i\ne j\\
 \end{cases} \, \text{ and }\, Y= \npmatrix{0&y \\
  y&0}.$$
Here $a_{ii}$, $a_{ij}$, $b_{ij}$, and $y$ are arbitrary nonzero real numbers.
  Since $X \circ I_{2n}=X \circ A=0$ and $[A,X]=0$, we have showed that $K_{2n}^{-n}\notin \GG$.
   \end{example}

Next, we prove that whether or not $P_n^c$ belongs to $\GG$ depends on divisibility of $n$ by $3$.  

\begin{theorem}\label{thm:PnC}
  $P_n^c \in \GG$ if and only if $n$ is not divisible by 3.
\end{theorem}

\begin{proof}
%
First let us assume that $n$ is not divisible by $3$. For $G=P_n^c$ we denote $V(G)=\{1,2,\ldots,n\}$ 
 and $E(G)=\{\{i,j\}; 1 \leq i \leq j-2 \leq n-2\}$.

 For $G=G_0$ and $l =1,2,\ldots, \left\lfloor \frac{n}{3} \right\rfloor$ define  
 $G_{l}=G_{l-1}+ \{3l,3l+1\}$ 
since by Lemma \ref{thm:rule1} we have
 $$\{3l-2,3l\} \forces{G}{G_{l-1}} \{3 l,3l+1\}\text{ for }l =1,2,\ldots, \left\lfloor \frac{n}{3} \right\rfloor.$$
 
 Let $m= \left\lfloor\frac{n}{3} \right\rfloor$.  
In the case, when $n=3m+2$, we construct 
$G_{m+k}=G_{m+k-1}\cup \{(n-3k-1,n-3k)\}$ for $k=1,2,\ldots,m$ using the fact that
$$\{n-3k+1,n-3k+3\}\forces{G}{G_{m+k-1}}  \{n-3k,n-3k+1\}
.
$$  
 We now observe that the pair $\{3t,3t+1\}$ is focused on $\{3t+2\}$ with respect to $G$ and $G_{2m}$, hence
   $$\{3t, 3t+1\}\forces{G}{G_{2m}}  \{3t+1, 3t+2\}$$
and so $G_{3m+1}=K_n$.
Similarly, in the case $n=3m+1$, for $s=0,\ldots, m-1$ the pair $\{n, 3s+1\}$ is focused on $2s+2$ with respect to $G$ and $G_{m}$,
 giving us $$\{n, 3s+1\} \forces{G}{G_{m}} \{3s+1, 3s+2\}.$$  
Let us denote the resulting graph in the sequence by $G_{l_1}$. Now, for $s=1,2,\ldots, m$, the pair $\{n,3s+2\}$ is focused on $\{3s+3\}$ with respect to $G$ and $G_{2m}$, hence
$$\{n, 3s+2\} \forces{G}{G_{2m}} \{3s+2, 3s+3\}$$
and $G_{l_2}=K_n$.
By Theorem \ref{sspseq} it follows that $G \in \GG$.


For $n=3m$, we will construct matrices $M\in \S(P_{3m}^c)$ and   $X \in\Szb(P_{3m})$ such that  $M\circ X=I_{3m} \circ X=0$ and $[M,X]=0$. This will show that $M\in  \S(P_{3m}^c)$ does not have the SSP and thus
 $P_{3m}^c \notin \GG$.

First, we pick  $\theta_1,\ldots,\theta_m$  such that $|\sin\theta_k|\neq |\sin\theta_{k+1}|$ for any $k=1,\ldots, m-1$ and $\sin\theta_k,\cos\theta_k\neq 0$ for any $k=1,\ldots, m$, and construct the matrix 
$X=X_{\theta_1}\oplus X_{\theta_2} \oplus \ldots \oplus X_{\theta_m}\in\Szb(P_{3m})$, where
\[X_\theta=\begin{bmatrix}
0 & \sin\theta & 0 \\
\sin\theta & 0 & \cos\theta \\
0 & \cos\theta & 0
\end{bmatrix}.\]

We define $M \in \S_{3m}(\R)$ to be a block matrix $M= \begin{bmatrix}M_{i,j}\end{bmatrix}_{i,j =1,\ldots,m}$ with blocks
 \[M_{i,j}=a_{i,j}\bx_{\theta_i}\bx_{\theta_j}\trans +b_{i,j}\by_{\theta_i}\by_{\theta_j}\trans +c_{i,j}\bz_{\theta_i}\bz_{\theta_j}\trans \in \S_3(\R),\] 
where
  \[\bx_\theta=\begin{bmatrix}
\cos\theta \\ 0 \\ -\sin\theta 
\end{bmatrix},
\by_\theta=\frac{1}{\sqrt{2}}\begin{bmatrix}
\sin\theta \\ 1 \\\cos\theta 
\end{bmatrix}, \text{ and }
\bz_\theta=\frac{1}{\sqrt{2}}\begin{bmatrix}
\sin\theta \\ -1 \\\cos\theta 
\end{bmatrix},\] 
and $a_{ij},$ $b_{ij}$ and $c_{ij}$ are the elements of some appropriately chosen $m \times m$ symmetric matrices $A$, $B$ and $C$. 
Note that any matrix $M$ defined in this way is symmetric, and since $X_\theta\bx_\theta=0$, $X_\theta\by_\theta=\by_\theta$ and $X_\theta\bz_\theta=-\bz_\theta$, it is easy to prove that $X_{\theta_i} M_{i,j}=M_{i,j}X_{\theta_j}$ for
all $i,j=1,2,\ldots,m$, which implies $[M,X]=0$.

Next let us look at conditions on matrices $A$, $B$ and $C$ that are required for the matrix $M$ to have zero elements required by the pattern of $G$. The diagonal blocks $M_{k,k}$ need to have the elements in the first row and second column, as well as the elements in the second row and third column equal to zero. This is achieved by choosing  $b_{k,k}=c_{k,k}$.  Furthermore, the entry in the third row and the first column of $M_{k,k+1}$ needs to be equal to $0$. This means that we need to choose $A$, $B$ and $C$ is such a way that the equality  
\begin{equation}\label{ak,k+1}
 -\cos\theta_k\sin\theta_{k+1}a_{k,k+1}+\frac{1}{2}\sin\theta_k\cos\theta_{k+1}(b_{k,k+1}+c_{k,k+1})=0
\end{equation}
is satisfied for $k=1,\ldots,m-1$.


%
%
To assure that we have all the required nonzero elements in $M$, we need to make sure that elements of $A$, $B$ and $C$ do not satisfy some finite set of linear equations, so generically chosen matrices $A$, $B$ and $C$ satisfying conditions indicated above, will result in a matrix $M \in  \S(P_{3m}^c) $.
\end{proof}

%

%

Next we prove that the graphs $G$ satisfying $q(G) \geq |G|-1$ have the SSP. 
Those graphs are characterised in \cite[Theorem 4.14]{MR3665573}, where it is proved that all such connected graphs fall into one of the following categories: 
\begin{itemize}
\item  paths,
\item  paths with one leaf attached to an interior vertex, or
\item paths with an extra edge joining two vertices at a distance $2$.
\end{itemize}
The graphs of the first two families are trees with at most one degree three vertex and thus by Theorem \ref{thm:trees} have the SSP.  To complete the question for this family, we prove that the graphs of the latter family have the SSP as well.

\begin{theorem}\label{thm:q=n-2}
Any connected graph $G$ satisfying $q(G) \geq |G|-1$ belongs to $\GG$.
\end{theorem}

\begin{proof}

Assume that $G=G_{n,m}$ is a graph that contains a path on vertices $\{1,2,\ldots,m,n,m+1,m+2,\ldots,n-1\}$ and an extra edge $\{m,m+1\}$. 
First, we apply Corollary \ref{cor:corofcor} to induced paths $P'$ with $V(P')=\{1,2,\ldots,m\}$ and 
   $P''$ with $V(P'')=\{n-1,n-2,\ldots,m+1\}$.
Let $G_1$ denote the graph with $V(G_1)=V(G)$, and $E(G_1)$ containing the edges of the complete 
graph on vertices of $P'$ and the edges of the complete graph on vertices of $P''$. 


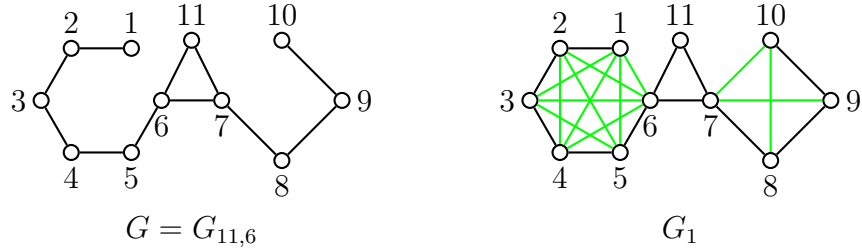
\begin{center}
\begin{tabular}{ccc}
\begin{tikzpicture}
\foreach \i in {1,2} {
\pgfmathsetmacro{\ang}{60*\i}
\node [label={above:$\i$}] (\i) at (\ang:1) {};
}
\foreach \i in {4,5} {
\pgfmathsetmacro{\ang}{60*\i}
\node [label={below:$\i$}] (\i) at (\ang:1) {};
}
\foreach \i in {3} {
\pgfmathsetmacro{\ang}{60*\i}
\node [label={left:$\i$}] (\i) at (\ang:1) {};
}
\foreach \i in {6} {
\pgfmathsetmacro{\ang}{60*\i}
\node [label={below:$\i$}] (\i) at (\ang:1) {};
}
\foreach \i in {7}{
\pgfmathsetmacro{\ang}{90*(\i-5)}
\node [label={below:$\i$}]  (\i) at ([xshift=3cm]\ang:1) {};
}
\foreach \i in {9}{
\pgfmathsetmacro{\ang}{90*(\i-5)}
\node [label={right:$\i$}]  (\i) at ([xshift=3cm]\ang:1) {};
}

\foreach \i in {8}{
\pgfmathsetmacro{\ang}{90*(\i-5)}
\node [label={below:$\i$}] (\i) at ([xshift=3cm]\ang:1) {};
}
\foreach \i in {10}{
\pgfmathsetmacro{\ang}{90*(\i-5)}
\node [label={above:$\i$}] (\i) at ([xshift=3cm]\ang:1) {};
}
\node[label={above:$11$}]  (11) at (1.5,1) {};
\draw (1) -- (2) -- (3) -- (4) -- (5) -- (6) -- (7) -- (8) -- (9) -- (10);
\draw (6) -- (11) -- (7);
    \end{tikzpicture} 
    &\qquad \qquad&
\begin{tikzpicture}
\foreach \i in {1,2} {
\pgfmathsetmacro{\ang}{60*\i}
\node [label={above:$\i$}] (\i) at (\ang:1) {};
}
\foreach \i in {4,5} {
\pgfmathsetmacro{\ang}{60*\i}
\node [label={below:$\i$}] (\i) at (\ang:1) {};
}
\foreach \i in {3} {
\pgfmathsetmacro{\ang}{60*\i}
\node [label={left:$\i$}] (\i) at (\ang:1) {};
}
\foreach \i in {6} {
\pgfmathsetmacro{\ang}{60*\i}
\node [label={below:$\i$}] (\i) at (\ang:1) {};
}
\foreach \i in {7}{
\pgfmathsetmacro{\ang}{90*(\i-5)}
\node [label={below:$\i$}]  (\i) at ([xshift=3cm]\ang:1) {};
}
\foreach \i in {9}{
\pgfmathsetmacro{\ang}{90*(\i-5)}
\node [label={right:$\i$}]  (\i) at ([xshift=3cm]\ang:1) {};
}
\foreach \i in {8}{
\pgfmathsetmacro{\ang}{90*(\i-5)}
\node [label={below:$\i$}] (\i) at ([xshift=3cm]\ang:1) {};
}
\foreach \i in {10}{
\pgfmathsetmacro{\ang}{90*(\i-5)}
\node [label={above:$\i$}] (\i) at ([xshift=3cm]\ang:1) {};
}
\node[label={above:$11$}]  (11) at (1.5,1) {};
\draw (1) -- (2) -- (3) -- (4) -- (5) -- (6) -- (7) -- (8) -- (9) -- (10);
\draw (6) -- (11) -- (7);

\begin{scope}[green]
\foreach \i in {1,...,4}{
    \pgfmathsetmacro{\ipp}{int(\i+2)}
    \foreach \j in {\ipp,...,6}{
        \draw (\i) -- (\j);
    }
}
\foreach \i in {7,8}{
    \pgfmathsetmacro{\ipp}{int(\i+2)}
    \foreach \j in {\ipp,...,10}{
        \draw (\i) -- (\j);
    }
}
\end{scope}
    \end{tikzpicture}
    \\
 $G=G_{11,6}$ && $G_1$
\end{tabular}
\end{center}

Observe that $$\{m,n\} \forces{G_1}{G} \{m-1,n\} \; \text{ and } \;\{m+1,n\} \forces{G_1}{G} \{m+2,n\}$$ 
by Lemma \ref{thm:rule1}, and denote
 $G_{2,1}=G_1 + \{\{m-1,n\}, \{m+1,n\}\}$.
Furthermore, $$\{m-1,m\}\forces{G}{G_{2,1}} \{m-1,m+1\}\; \text{ and } \; \{m+1,m+2\}\forces{G}{G_{2,1}} \{m,m+2\}$$ and let 
$G_{2}=G_{2,1} + \{\{m-1,m+1\}, \{m,m+2\}\}$. Note that $G^{(2)}$ is a subgraph of $G_2$. 

We continue by observing that  $$\{m-1,n\}\forces{G}{G_{2}}  \{m-2,n\} \; \text{ and } \;\{m+2,n\}\forces{G}{G_{2}} \{m+3,n\},$$
resulting in graph $G_{3,1}$. Now 
$$\{m-2,m\}\forces{G}{G_{3,1}} \{m-2,m+1\}\; \text{ and } \; \{m+1,m+3\}\forces{G}{G_{3,1}} \{m,m+3\},$$
resulting in graph $G_{3,2}$, and finally 
$$\{m-1,m+1\}\forces{G}{G_{3,2}} \{m-1,m+2\},$$
 resulting in graph $G_3$,  which has $G^{(3)}$ as a subgraph. 

We proceed to use Lemma \ref{thm:rule1} in this way to obtain graphs $G_4 \subset G_5 \subset \cdots \subset G_{n-1}=K_n$, where each
$G_k$ contains $G^{(k)}$ as a subgraph.   By Theorem \ref{sspseq} this proves $G=G_{n,m} \in \GG$.
\end{proof}
 
The final example in this section illustrates the application of Corollary \ref{cor:corofcor}.  

\begin{example}\label{lollipop}
 Let $G=L_{m,n}$ be an $(m,n)$-lollipop graph, i.e.\ the graph with $m+n$ vertices obtained by 
 joining a complete graph $K_m$ to a path graph $P_n$ by a bridge.  
Hence,  $V(G)=\{1,2,\ldots,m+n\}$ and 
 $$E(G)=\{\{i,j\}; \, 1 \leq i<j \leq n\} \cup \{\{i,i+1\}; \, n\leq i \leq n+m-1\}.$$

We apply Corollary \ref{cor:corofcor} to 
 an induced path $P'$ with $V(P')=\{n,n+1,\ldots,n+m\}$ and denote by $G_1$ the graph with $V(G_1)=V(G)$, 
 and $E(G_1)$ containing the edges of $E(G)$ as well as those of the complete 
graph on vertices of $P'$. 

Observe that for $i_1=1,2,\ldots,n$, we have 
$$\{i,n\}\forces{G}{G_1}  \{i,n+1\}$$ and let 
$G_{2}=G_{1}+ \{\{i,n+1\}\colon \; 1 \leq i \leq n-1\}$.
Now, consecutively for $j=2,3,\ldots,m$ we have
$$\{i,n+j-1\}\forces{G}{G_j} \{i,n+j\}$$ 
and define $G_{j+1} = G_{j}+ \{\{i,n+j\}\colon \; 1 \leq i \leq n-1\}$.
Since $G_{n+m+1}=K_{n+m}$, it follows that  $L_{m,n} \in \GG$ by Theorem \ref{sspseq}.
\end{example}

Now that the characterization of trees in $\GG$ is known, a question which unicyclic graphs are in $\GG$ is next to consider. 
Corollary \ref{rem:uninotinGG} excludes a large family of unicyclic graphs from $\GG$, such as any unicyclic graph with a vertex of degree at least four (see e.g.~$G_{96}$ on Figure \ref{G96}) and all unicyclic graphs with a degree three vertex not contained in the cycle (see e.g.~$H$ on Figure \ref{H}).

\begin{center}
 \begin{figure}[htb!]
 \centering
   \subfigure[$G_{96}$]{
\begin{tikzpicture}[style=thick, scale=0.8]   \label{G96}
	 \coordinate (1)   at (0:1);
	 \coordinate (2)   at (90:1);
	 \coordinate (3)   at (180:1);
	 \coordinate (4)   at (270:1);
	 \coordinate (5)   at (2,1);
	 \coordinate (6)   at (2,-1);
  	 \draw (1) -- (2) --(3)--(4) -- (1)--(5);
	 \draw (1) -- (6);
	 \draw[fill=white] \foreach \x in {(1),(2),(3),(4),(5),(6)} {
		                    \x circle (1mm)
		};	
    \end{tikzpicture}
}
   \subfigure[$G_{98}$]{
  \begin{tikzpicture}[style=thick, scale=0.8]    \label{G98}
	 \coordinate (1)   at (-1,1);
	 \coordinate (2)   at (-1,-1);
	 \coordinate (3)   at (1,-1);
	 \coordinate (4)   at (1,1);
	 \coordinate (5)   at (2,1);
	 \coordinate (6)   at (2,-1);
 	 \draw (4) -- (1) -- (2) --(3)--(4) -- (5);
  	 \draw (3) -- (6);
	 \draw[fill=white] \foreach \x in {(1),(2),(3),(4),(5),(6)} {
		                    \x circle (1mm)
		};	
    \end{tikzpicture}   
}
    \subfigure[$G_{99}$]{
 \begin{tikzpicture}[style=thick, scale=0.8]    \label{G99}
	 \coordinate (1)   at (0:1);
	 \coordinate (2)   at (90:1);
	 \coordinate (3)   at (180:1);
	 \coordinate (4)   at (270:1);
	 \coordinate (5)   at (0:2);
	 \coordinate (6)   at (180:2);
  	 \draw (1) -- (2) --(3)--(4) -- (1)--(5);
	 \draw (3) -- (6);
	 \draw[fill=white] \foreach \x in {(1),(2),(3),(4),(5),(6)} {
		                    \x circle (1mm)
		};	
    \end{tikzpicture}
    }
      \subfigure[$H$]{
   \begin{tikzpicture}[style=thick, scale=0.8]    \label{H}
	 \coordinate (1)   at (0:1);
	 \coordinate (2)   at (90:1);
	 \coordinate (3)   at (180:1);
	 \coordinate (4)   at (270:1);
	 \coordinate (5)   at (2,0);
	 \coordinate (6)   at (3,1);
 	 \coordinate (7)   at (3,-1);
  	 \draw (1) -- (2) --(3)--(4) -- (1)--(5)--(6);
	 \draw (5) -- (7);
	 \draw[fill=white] \foreach \x in {(1),(2),(3),(4),(5),(6),(7)} {
		                    \x circle (1mm)
		};	
    \end{tikzpicture}
}
       \caption{Examples of unicyclic graphs not in $\GG$.}
\end{figure}
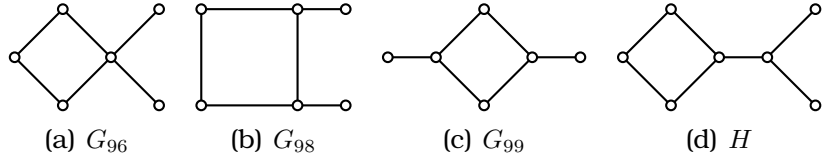
\end{center}

The next example features further examples of unicyclic graphs with two degree three vertices, that are not in $\GG$.

\begin{example}
Consider the graphs $G_{98}$ and $G_{99}$ as shown on Figures \ref{G98} and \ref{G99}.
Note that for matrices
\begin{align*}
A_{98}&=
\left(
\begin{array}{cccccc}
 0 & 2 & 0 & 1 & 0 & 0 \\
 2 & 0 & 1 & 0 & 0 & 0 \\
 0 & 1 & 0 & 1 & 1 & 0 \\
 1 & 0 & 1 & 0 & 0 & 1 \\
 0 & 0 & 1 & 0 & 0 & 0 \\
 0 & 0 & 0 & 1 & 0 & 0 \\
\end{array}
\right), & X_{98}&=
 \left(
\begin{array}{cccccc}
 0 & 0 & 1 & 0 & 0 & 1 \\
 0 & 0 & 0 & 1 & 1 & 0 \\
 1 & 0 & 0 & 0 & 0 & -1 \\
 0 & 1 & 0 & 0 & -1 & 0 \\
 0 & 1 & 0 & -1 & 0 & 0 \\
 1 & 0 & -1 & 0 & 0 & 0 \\
\end{array}
\right),
\\
A_{99}&=
 \left(
\begin{array}{cccccc}
 0 & 1 & 0 & 1 & 0 & 0 \\
 1 & 0 & 1 & 0 & 1 & 0 \\
 0 & 1 & 0 & -1 & 0 & 0 \\
 1 & 0 & -1 & 0 & 0 & 1 \\
 0 & 1 & 0 & 0 & 0 & 0 \\
 0 & 0 & 0 & 1 & 0 & 0 \\
\end{array} 
\right),& X_{99}&=
\left(
\begin{array}{cccccc}
 0 & 0 & 0 & 0 & 1 & 1 \\
 0 & 0 & 0 & 1 & 0 & 0 \\
 0 & 0 & 0 & 0 & -1 & 1 \\
 0 & 1 & 0 & 0 & 0 & 0 \\
 1 & 0 & -1 & 0 & 0 & -1 \\
 1 & 0 & 1 & 0 & -1 & 0 \\
\end{array}
\right),
\end{align*}
where $A_i\in \S(G_i)$ and $X_i\in \Szb(G_i^c)$, 
we have $A_i \circ X_i=I_6 \circ X_i=0$ and $[A_i,X_i]=0$, which proves $G_{98}, G_{99} \notin \GG$.
\end{example}

In Example \ref{smallgraph} we presented a unicyclic graph in $\GG$ with one degree three vertex, however, the question of determining which other graphs in this family belong to $\GG$ is still to be resolved.

\section*{Acknowledgement}
Jephian C.-H. Lin was partially supported by MOST grant 107-2115-M-110-008-MY2.

\bibliographystyle{plain}
\bibliography{ssp}

\end{document}